\documentclass{article}
\pdfoutput=1
\usepackage{arxiv}
\usepackage[utf8]{inputenc} 
\usepackage[T1]{fontenc}    
\usepackage{hyperref}       
\usepackage{url}            
\usepackage{booktabs}       
\usepackage{amsfonts}       
\usepackage{nicefrac}       
\usepackage{microtype}      
\usepackage{adjustbox}
\usepackage{makecell}
\usepackage{graphicx}
\usepackage{doi}
\usepackage{tipa}
\usepackage{mathrsfs}
\usepackage{mathtools}
\usepackage{algorithm}
\usepackage{cancel}
\usepackage{amsthm}
\usepackage{underoverlap}
\usepackage{stmaryrd}
\usepackage{tikz}			                
\usepackage{tikz-qtree}			                
\usetikzlibrary{shapes.geometric}
\usetikzlibrary{arrows.meta,arrows}
\usetikzlibrary{3d,tikzmark,intersections,scopes,shapes,shapes.multipart,calc,positioning,fit,bending,backgrounds,decorations.markings,cd}
\usepackage{gb4e}						
\noautomath

\newtheoremstyle{mystyle}
{3pt}
{3pt}
{\itshape}
{}
{\bfseries}
{.}
{.5em}
{}

\theoremstyle{mystyle}
\newtheorem{definition}{Definition}[section]
\newtheorem{theorem}{Theorem}[section]
\newtheorem{corollary}{Corollary}[section]
\newtheorem{lemma}{Lemma}[section]
\newtheorem{example}{Example}[section]

\newtheorem*{remark}{Remark}
\renewenvironment{proof}{{\noindent\bfseries Proof:}}{\qed}
\makeatletter
\renewenvironment{proof}[1][\proofname]{%
	\par\pushQED{\qed}\normalfont%
	\topsep6\p@\@plus6\p@\relax
	\trivlist\item[\hskip\labelsep\bfseries#1\@addpunct{.}]%
	\ignorespaces
}{%
	\popQED\endtrivlist\@endpefalse
}
\makeatother
\newcommand{\qte}[1]{`{#1}'}					
\newcommand{\term}[1]{{\bfseries #1}}				
\newcommand{\et}{\,{\textnormal{and}}\,}			
\newcommand{\siv}{\,{\textnormal{if}}\,}			
\newcommand{\tun}{\,{\textnormal{then}}\,}			
\newcommand{\wenn}{\,{\textnormal{if and only if}}\,}		
\newcommand{\wennif}{\,{\textnormal{iff}}\,}			
\newcommand{\eg}{\textit{e.g.}}					
\newcommand{\ie}{\textit{i.e.}}					
\newcommand{\etc}{\textit{etc.}}				
\newcommand{\dft}{\textit{de facto}\,}				
\newcommand{\sans}{\textit{sans}\,}				
\newcommand{\avec}{\textit{avec}\,}				
\newcommand{\orth}[1]{\bt{\textnormal{#1}}}					
\newcommand{\ol}[1]{\textit{#1}} 					
\newcommand{\com}{,\,}						
\newcommand{\f}[1]{\ensuremath{#1}} 				
\newcommand{\fun}[2]{{\ensuremath{#1(~{#2}~)}}}			
\newcommand{\funt}[3]{{\ensuremath{#1(~{#2}~)(~{#3}~)}}}	
\newcommand{\funeq}[3]{{\ensuremath{#1(~{#2}~)=#3}}}		
\newcommand{\esc}[1]{\,{\textnormal{#1}}\,}					
\newcommand{\bc}[2][]{\ensuremath{\left\lbrace~{#2}~\right\rbrace_{#1}}} 	
\newcommand{\bcdef}[2]{\ensuremath{\left\lbrace~{#1} \mid {#2}~\right\rbrace}} 	
\newcommand{\bt}[2][]{\ensuremath{\left\langle~{#2}~\right\rangle_{#1}}}      	
\newcommand{\bp}[2][]{\ensuremath{\left(~{#2}~\right)_{#1}}}                  	
\newcommand{\bq}[2][]{\ensuremath{\left[~{#2}~\right]_{#1}}}             	
\newcommand{\norm}[2][]{\ensuremath{\left\lVert~{#2}~\right\rVert_{#1}}}     	
\newcommand{\sph}{\ensuremath{\varphi}}							
\newcommand{\sps}{\ensuremath{\psi}}							
\newcommand{\mph}[2][ ]{\ensuremath{{#2}_{#1}}}						
\newcommand{\dom}[1]{\ensuremath{\textnormal{\normalfont\textsc{dom}}\bp{#1}}}		
\newcommand{\ran}[1]{\ensuremath{\textnormal{\normalfont\textsc{codom}}\bp{#1}}}	
\newcommand{\img}[1]{\ensuremath{\textnormal{\normalfont\textsc{im}}\bp{#1}}}	
\newcommand{\mon}{\ensuremath{\textnormal{\normalfont\textsc{mon}}}}				
\newcommand{\dya}{\ensuremath{\textnormal{\normalfont\textsc{dya}}}}				
\newcommand{\tri}{\ensuremath{\textnormal{\normalfont\textsc{tri}}}}				
\newcommand{\nry}{\ensuremath{\textnormal{\normalfont\textsc{nop}}}}				
\newcommand{\cond}{\ensuremath{\textnormal{\normalfont\textsc{cond}}}}				
\newcommand{\plx}{\ensuremath{\textnormal{\normalfont\textsc{c}}}}				
\newcommand{\maj}[1]{\ensuremath{\textnormal{\normalfont\textsc{maj}}_{#1}}}				
\newcommand{\mdl}[2][ ]{\ensuremath{\mathcal{#2}_{#1}}}		
\newcommand{\gsn}{\ensuremath{g}}				
\newcommand{\gsnt}[1]{\ensuremath{g\left[~{#1}~\right]}}	
\newcommand{\fs}{\ensuremath{\mathcal{FS}}}			
\newcommand{\fl}{\ensuremath{\mathcal{FL}}}			
\newcommand{\nl}{\ensuremath{\mathcal{NL}}}			
\newcommand{\clc}{\ensuremath{\mathcal{C}}}			
\newcommand{\den}[2][]{{\ensuremath{\llbracket}\,{\esc{\textbf{#2}}}\,\ensuremath{\rrbracket}}\ensuremath{^{{#1}}}}	
\newcommand{\vv}[2][]{\ensuremath{\mathbf{#2}_{#1}}}
	
\newcommand{\ff}[2][]{\ensuremath{\mathbb{#2}^{#1}}} 				
\newcommand{\inn}[2]{\ensuremath{\bt{#1~|~#2}}}
\newcommand{\sml}{{\ensuremath{\textnormal{\normalfont\textsc{sim}}}}}	
\newcommand{\abs}[1]{\ensuremath{\left\lvert~{#1}~\right\rvert}}     
\newcommand{\ceil}[1]{\ensuremath{\left\lceil~{#1}~\right\rceil}}      

\title{A vector logic for extensional formal semantics}

\author{ \href{https://orcid.org/0009-0004-7957-1806}{\includegraphics[scale=0.06]{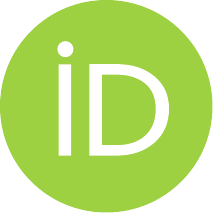}\hspace{1mm}Daniel Quigley} \\
	Department of Linguistics\\
	University of Wisconsin-Milwaukee\\
	Milwaukee, WI 53211 \\
	\texttt{quigleyd@uwm.edu} \\
}

\renewcommand{\headeright}{Preprint}
\renewcommand{\undertitle}{Preprint}
\renewcommand{\shorttitle}{A vector logic for extensional formal semantics}
\colorlet{linecol}{cyan!90!blue!90!black}
\colorlet{fillcol}{cyan!60!blue!80!black!40}
\hypersetup{
pdftitle={A vector logic for extensional formal semantics},
pdfsubject={linguistics, mathematics, logic},
pdfauthor={Daniel Quigley},
pdfkeywords={formal semantics, distributional semantics, extensional model, semantic space, vector logic, mathematical linguistics},
}

\begin{document}
\maketitle

\begin{abstract}\label{abs:abstract}

This paper proves a homomorphism between extensional formal semantics and distributional vector space semantics, demonstrating structural compatibility. Formal semantics models meaning as reference, using logical structures to map linguistic expressions to truth conditions, while distributional semantics represents meaning through word vectors derived from contextual usage. By constructing injective mappings that preserve semantic relationships, we show that every semantic function in an extensional model corresponds to a compatible vector space operation. This result respects compositionality and extends to function compositions, constant interpretations, and \f{n}-ary relations. Rather than pursuing unification, we highlight a mathematical foundation for hybrid cognitive models that integrate symbolic and sub-symbolic reasoning and semantics. These findings support multimodal language processing, aligning \qte{meaning as reference} (Frege, Tarski) with \qte{meaning as use} (Wittgenstein, Firth).

\end{abstract}


\section{Introduction}\label{sec:intro}
\term{Formal semantics} and \term{distributional semantics} represent two distinct mathematical approaches to representing meaning in language. The formal approach, developed through \cite{Montague1974,Gentzen1935}, maps linguistic expressions to truth values via \term{extensional} or \term{intensional} models \cite{cop2024,Montague2002}. In contrast, distributional semantics, following \cite{harrismathematical1968}'s hypothesis and popularized by \cite{Firth1957}'s \qte{you shall know a word by the company it keeps}, represents meaning through distributions in a high-dimensional \term{vector space} where semantic relationships are captured through proximity, agreement, and interpolation \cite{marcolli2023syntaxsemantics,Gardefors2014}. This distributional approach underpins modern language models and computational linguistics \cite{lenci2022comparative,boleda2020distributional}.

There exists a homomorphism between an extensional model of formal semantics with a vector space model of distributional semantics. A homomorphism between two classes of mathematical structures $C$ and $D$ means a function that preserves operations and relations between them: it maps elements from $C$ to $D$ while maintaining the structural relationships defined within $C$ in their corresponding images in $D$. More formally, for any operation or relation in $C$, applying it to elements and then mapping the result equals mapping the elements first and then applying the corresponding operation or relation in $D$. In the results of this paper, we prove that truth-theoretic semantics in the style of \cite{Heim1998HEISIG,vonFintelHeim1997}, built on the framework of a formal system, has a representation in and is compatible with a vector space model of semantics built on the framework of vector logic \cite{WestphalHardy2005,MIZRAJI1992179,mizr2008}.

This paper is not a soliloquy; while this is not the first attempt at such a brokering of peace between the two, it does represent the first (as far as the author is aware) formal proofs that a compatibility exists at all of their relationships as far as the processing of natural language is concerned. This is to say that there is precedence for this relationship between discrete and continuous, between the formal and the distributional, but a foundational proof remains. From a purely logical setting, \cite{mizr2008,MIZRAJI1992179,WestphalHardy2005} encode discrete logic into vector spaces, and form a crucial formalization necessary in applications such as quantum computing \cite{gudder2003quantum,dalla2005logics,muthukrishnan2000multivalued,chiara2003quantum}. \cite{herbelotvecchi2015building} goes the other way, from the distributional to the discrete. In \cite{Burnistov2023}, two different representations of formalizing and representing the semantics of natural language sentences (rank models and vector models) are shown to have a correspondence by invoking the notion of Scott-Ershov spaces from domain theory developed in \cite{Ershov1972Functionals,Ershov1972Lattices,Scott1972Toposes,Scott1970Theory}, to model computation; rank models and vector models are effectively interpretable in each other, and hence they have the same level of complexity. In broad strokes, this is similar to the main results of this paper, in which one model is shown related to the other; where we principally differ, however, is in this notion of computability spaces, an invocation we do not follow here.

Sections~\ref{sec:fs} and \ref{sec:vs} are exposition. Section~\ref{sec:fs} lays out in detail the formal semantics of a typed extensional model as relevant in linguistics and logic; here, we give a technical presentation in a concise and consistent way that is (at least, to the best of the author's knowledge) otherwise absent in standard texts. In Section~\ref{sec:vs}, we give a brief introduction to distributional semantics. We proceed with the main theorem of the paper, Theorem~\ref{thm:homo}, in Section~\ref{sec:homo} and its subsequent proofs. We deliberate on the consequences of this theorem as it relates to the principle of compositionality in Section~\ref{sec:comp}. A natural next step for future work is to expand Theorem~\ref{thm:homo} to an intensional model of semantics, which is introduced in Section~\ref{sec:intens}. We conclude in Section~\ref{sec:disc} with the philosophical issues of unification and mathematization as opposed to the notion of compatibility, a kind of pluralism, for which we advocate.

Some conventions used in this paper: a \term{Definition} is an explanation of a term or concept, and is written in bold typeface, and if such a definition is especially pertinent, it is given in a definition environment; a \term{Theorem} is a true statement that has not been stated elsewhere in the literature or yet proven to be true, and is given in a theorem environment; a \term{Lemma} is a true statement used in proving other true statements, and is given in a lemma environment; a \term{Proposition} is a \qte{less important} but nonetheless interesting true statement that may have been stated elsewhere in the literature but has not been proven therein, and is given in a proposition environment; a \term{Corollary} is a true statement that is a simple deduction from a theorem, lemma, or proposition; \term{Example} environments are used for non-linguistic examples; conventional linguistic left-hand-side-numbered environments are used primarily for linguistic data where appropriate. Finally, while accessibility features for \textit{.pdf} documents generated by \LaTeX{} are burgeoning, effort has been made to accommodate vision accessibility where possible throughout \cite{lars2023,tys2020}. To that end, most mathematical expressions involving closed brackets (parentheses or any of the various other bracket-types) are written with greater kerning \cite{dotan2018,beier2021,mcleish2007,sjob2016,BlackmoreWright2013}.

\section{Formal semantics}\label{sec:fs}
\term{Formal semantics}\footnote{Much of the following exposition follows from \cite{quigley2024categoricalframeworktypedextensional}.} refers to approaches to the study of meaning in language as a formal system, and includes model-theoretic \cite{Montague1974,Partee2008} and proof-theoretic semantics \cite{Gentzen1935}; the qualifier \qte{formal} refers to simply the manipulation of symbols according to a set of rules.

\begin{definition}[Formal System]\label{def:formalsys}
	A \term{formal system} is a triple of the form\begin{align*}
		\mdl[\fs]{M} = \bt{{\fl,\models,\vdash_{\clc}}},
	\end{align*}
	where \fl{} is a formal language, defined as the set of all well-formed formulas constructed according to appropriate formation rules; \f{\models} is \term{model-theoretic entailment}, a binary relation on the set of formulas of \fl; \f{\vdash_{\clc}} is \term{proof-theoretic derivability} in a calculus \f{\clc}, a binary relation on the set of formulas of \fl.
\end{definition}

The formal system is \term{sound} when: if \f{\Gamma \vdash_{\clc} \sph} (derivability), then \f{\Gamma \models \sph} (semantic entailment); the formal system is \term{complete} when: if \f{\Gamma \models \sph}, then \f{\Gamma \vdash_{\clc} \sph} (derivability). When both soundness and completeness hold, the semantic and proof-theoretic notions of entailment coincide, and we have a complete formal logic. To apply the formal system, we transform the \nl{} fragment into a computationally tractable fragment, \ie, into a formal language \fl, and interpret it relative to a model. We denote by \f{\models_{\nl}} the truth conditions and natural consequence relations on natural language utterances; if the calculus \f{\clc} of the logical system is adequate (\ie, sound and complete), then it is a model of the linguistic entailment relation \f{\models_{\nl}} observed in natural language.

Model theory is a branch of mathematical logic in which the relationship between a formal language and its interpretations are captured in a structure \mdl{M} called a \term{model} \cite{chang1990model,mendelson1997introduction,fagin2004reasoning}, which (minimally) contains the relevant elements and a method by which to interpret them. Classical model theory is the model theory of first-order (predicate) logic, from which are constructed the extensional and intensional semantics for natural language \cite{parteemathematical1990,Heim1998HEISIG}. Mainstream model theory is the model theory for either formal or natural language by invoking set-theoretic structures with Tarskian truth values \cite{Tarski1931TARTCO} and Montague grammar \cite{Montague1970b,Montague1974,Montague1970,Dowty1981}.

The following outlines the key components and conditions necessary for constructing a valid formal semantic system that models natural language meaning, starting with the translation of natural language into a formal language and proceeding through the requirements for soundness, completeness, and semantic preservation.

\begin{enumerate}
	\item Given a natural language fragment \f{nl \in \nl}, the translation \f{T: \nl \rightarrow \fl} is such that, for any natural language fragment \f{nl \in \nl}, \f{\funeq{T}{nl}{\sph}}, where \f{\sph \in \fl}.
	
	\item For any well-formed formula and a model \mdl{M}, \den[\mdl{M}]{\f{\sph}} assigns a semantic value to \f{\sph} in \mdl{M}.
	\begin{enumerate}
		\item \sph{} is \term{true} in the model \mdl{M} if the interpretation function \mdl{I} assigns the truth value \f{1} to \sph{} under the model \mdl{M}\begin{align*}
			\mdl{M} \models \sph \wennif \funeq{\mdl{I}}{\sph}{1}.
		\end{align*}
		
		\item \sph{} \term{entails} another \sps{} if for every model in which \sph{} is true, then \sps{} is also true \mdl{M}\begin{align*}
			\sph \models \sps \wennif \forall \mdl{M} \siv  \mdl{M} \models \sph\com \tun \mdl{M} \models \sps.
		\end{align*}	
	\end{enumerate}
	
	\item Choose a calculus \clc{} for the formal language \fl:
	\begin{enumerate}
		\item \clc{} is \term{sound} if, for a formula \sph{} derived from a set of axioms \f{\Lambda} using the rules of \clc{}, \sph{} is true in \mdl{M} whenever \f{\Lambda} is true in \mdl{M}\begin{align*}
			\forall\Lambda,\sph\com\siv\Lambda\vdash_{\clc} \sph\et\mdl{M}\models\Lambda\com\tun \mdl{M}\models\sph.
		\end{align*}
		
		\item \clc{} is \term{complete} if, whenever a formula \sph{} is a logical consequence of a set of axioms \f{\Lambda} (\ie, \f{\Lambda\models\sph}), then there exists a formal derivation of \f{\sph} from \f{\Lambda} in the calculus \clc{} (\ie, \f{\Lambda \vdash_\clc \sph})\begin{align*}
			\forall\Lambda\com\siv\Lambda\models\sph\com\tun\Lambda\vdash_{\clc}\sph.
		\end{align*}
	\end{enumerate}
	\item If the calculus \clc{} is both sound and complete, then we have a valid formal system.
	
	\item If our goal is to characterize the truth conditions and entailment relations of natural language fragments, then:
	\begin{enumerate}
		\item If the translation function \f{T} preserves semantic structure (\ie, the truth conditions and entailments in \nl{} correspond to those of the translated formulas in \fl), 
		
		\begin{align*}
			\forall \nl_1,\nl_2\in\nl\com\nl_1\models_\nl \nl_2 \wennif \fun{T}{\nl_1}\models\fun{T}{\nl_2},
		\end{align*}

		and if the calculus \clc{} adequately captures the \nl{} entailment relation \f{\models_\nl}, 
		
		\begin{align*}
			\forall \nl_1,\nl_2\in\nl\com\nl_1\models_\nl \nl_2 \wennif \fun{T}{\nl_1}\vdash_\clc \fun{T}{\nl_2},
		\end{align*}
		
	\end{enumerate}
	then the formal system \fs{} models the semantics of the relevant \nl{} fragment.
	
	\item If the formal system \fs{} models the semantics of the relevant \nl{}  fragment, then we have a valid model of natural language.

\end{enumerate}

The \dft calculus for formal semantics is a simply-typed lambda calculus \cite{Dowty1981,Heim1998HEISIG,Montague2002}; that is, the lambda calculus of \cite{Church1932} equipped with labels on the domain \cite{Church1940}. These labels (\term{types}) give restrictions on the components that may be composed together \cite{Alma2013,Heim1998HEISIG}, formulated as an answer to a variant of Russell's paradox called the Rosser-Kleene paradox \cite{KleeneRosser1935,Curry1946}. The simply typed lambda calculus of Heim-Kratzer style semantics \cite{Heim1998HEISIG,vonFintelHeim1997} is the formal system used here when discussing \qte{formal semantics}\footnote{One assumes also that there is a computational system of syntax that produces, as input for semantic interpretation, some tree structures; we assume a context free grammar \cite{parteemathematical1990,schubert1982english,zafar2012formal,Heim1998HEISIG}.}.

\subsection{Extensional model of semantics}\label{subsec:ext}

Mathematically, functions are either: (1) graphs; (2) formulae. The functions-as-graphs is the extensional interpretation of a function; the functions-as-formulae is the intensional interpretation of a function. Consider the \qte{functions-as-graphs} correspondence. In this way, functions are subsets of their domain-codomain product space; that is, a function from a domain \f{\dom{X}} to a codomain \f{\ran{Y}} is given by \f{f:X\rightarrow Y}, and is a subset of the product space \f{f\subset X\times Y}. Two functions \f{f:X\rightarrow Y} and \f{g:X\rightarrow Y} are identical if they map equal inputs to equal outputs. This defines the \term{extensional} interpretation of a function: a function is \emph{fully determined} by what are its input-output relationships \cite{Selinger2013}. Extensional semantics, then, is concerned with all elements in the set, in which the extension of an expression is the set of all elements that the expression refers to or the truth value to which it evaluates.

Extensional formal semantics is type-driven; every meaningful thing knows exactly what kind of thing it is and what it can combine with. Every semantic element carries three essential pieces of information: (1) its fundamental nature (encoded in its type); (2) its potential for action (shown by what inputs it accepts, if any, and its outputs, if any); (3) its combinatorial possibilities (determined by how it can compose with other types). The type system ensures that only sensible combinations can occur, ruling out nonsensical constructions automatically. 

\begin{definition}[Type]\label{def:exttyp}
	\term{Types} are labels on the domain that specify what kind of object something is and what it can combine with. For abstract types \bc{\alpha, \beta}:
	
	\begin{itemize}
		\item if \f{\alpha} is a type and \f{\beta} is a type, then \f{\alpha\times\beta} is the type of an ordered pair whose first element is of type \f{\alpha} and whose second element is of type \f{\beta};
		\item if \f{\alpha_1, \alpha_2,\dots,\alpha_n} are types, then \f{\alpha_1\times \alpha_2\times\dots\times\alpha_n} is the type of ordered \f{n}-tuples, where each element is of the corresponding type \f{\alpha_i};
		\item if \f{\alpha} is a type and \f{\beta} is a type, then \bt{\alpha,\beta} is the type of functions from \f{\alpha} to \f{\beta}, whose domain is the set of entities of type \f{\alpha} and whose range is the set of entities of type \f{\beta};
		\item if \f{\alpha} is a type, then \fun{\mdl{P}}{\alpha} is the type of sets containing elements of type \f{\alpha};
		\item nothing else is a type.
	\end{itemize}
\end{definition}

Extensional semantics has the two primitive types: \bc{e,t}. \f{e} is the type of \term{entities}, all entities have type \f{e}; \f{t} is the type of \term{truth values}, the truth values \bc{1,0} have type \f{t}. 
\begin{align*}
	e &: \text{entities in the world} \\
	t &: \text{truth or false} \in \bc{1,0}
\end{align*}

Consider the noun \ol{cat}; fundamentally, if \ol{cat} refers to an individual, it is of type \f{e}, and if it refers to the set of cats, it is of type \bt{e,t}. The adjective \ol{sad} is a one-place predicate describing a property (\eg, being sad), and takes an individual and returns a truth value, so is of type \bt{e,t}, checking if its argument has the property of sadness. The transitive verb \ol{chase} is a two-place predicate, representing a relation between two individuals (\eg, the chaser and the chased) to make a true or false statement, and so is of type \bt{e,\bt{e,t}}.

\begin{definition}[Extensional model]\label{def:extmod}
	Let \f{\mathcal{T}} be a set of extensional types. An \term{extensional model} \mdl[x]{M} is a tuple of the form:\begin{align*}
		\mdl[x]{M} = \bt{\bp[{\tau\in\mathcal{T}}]{\mdl[\tau]{D}},\mdl{I}},
	\end{align*}
	where:
	\begin{itemize}
		\item For each type \f{\tau\in\mathcal{T}}, the model determines a corresponding domain \mdl[{\tau}]{D}. The \term{standard domain} or \term{domain of discourse} is an indexed family of sets \bp[{\tau\in\mathcal{T}}]{\mdl[\tau]{D}} for all types \f{\tau\in\mathcal{T}}:
		\begin{itemize}
			\item \f{\mdl[e]{D}=\mdl{D}} is the set of entities.
			\item \f{\mdl[t]{D}=\f{\bc{0,1}}} is the set of truth-values.
		\end{itemize}
		\item For each type \f{\tau\in\mathcal{T}} and each non-logical constant \f{c_\tau} of type \f{\tau} in the formal language \fl, the extensional \term{interpretation function} \mdl{I} assigns an element from the corresponding domain \mdl[\tau]{D}, such that \f{\fun{\mdl{I}}{c_\tau}\in\mdl[\tau]{D}}.

		Let \f{V} be the set of individual variables, and \mdl[e]{D} be the domain of entities. The \term{assignment function}\footnote{\textnormal{The semantic value of an expression is \emph{relativized} to a model and an assignment; the same expression can have different semantic values depending on the specific model and assignment we consider, hence it is not an element of the model tuple. Interestingly, \cite{parteemathematical1990} does include \gsn{} in the model tuple (as well as integer indices \f{i} and \den{.} explicitly); the treatment here follows the conventions from \cite{Sider2009SIDLFP,Chierchia2000CHIMAG,kohl2024,cop2024,Gallin1975}, in which \gsn{} is not an element of the model tuple.}} \f{\gsn:V\rightarrow\mdl[e]{D}} maps each variable \f{x\in V} to an entity \f{\fun{\gsn}{x}\in\mdl[e]{D}}. For any variable \f{x\in V} and individual \f{k\in \mdl[e]{D}}, the \term{variant assignment function} \gsnt{x\mapsto k} is defined by:\begin{align*}
			\fun{\gsnt{x\mapsto k}}{y}= \begin{cases}k & \text { if } y=x, \\ \fun{\gsn}{y} & \text { if } y \neq x .\end{cases}
		\end{align*}
		\begin{enumerate}
			\item If \sph{} is a constant \f{\sph \in \fl\com \funeq{\mdl{I}}{\sph}{k} \in\mdl[e]{D}}.
			
			\item If \sph{} is a variable \f{\sph \in V\com \funeq{\gsn}{\sph}{k} \in \mdl[e]{D}}.
			
			\item For any \f{n}-ary predicate \f{P \in\fl},\space \f{\fun{\mdl{I}}{P}\subseteq\mdl[e]{D}^n} (\fun{\mdl{I}}{P} is the extension of \f{P} a subset of the \f{n}-th Cartesian power of \mdl[e]{D}).
			
			\item For any \f{n}-ary function \f{f:\mdl[e]{D}^n\rightarrow\mdl[e]{D}},\space \f{\fun{\mdl{I}}{f}: \mdl[e]{D}^n\rightarrow\mdl[e]{D}}.
		\end{enumerate}
		
		An extensional \term{denotation function}\textnormal{\footnote{Superscripts in the double square brackets (\qte{semantic brackets}) specifies the interpretation function it is based on (\ie, extensional or intensional, \sans or \avec an assignment function, \etc).}} \den[\mdl{M},\gsn]{.} assigns to every expression \sph{} of the language \fl{} a semantic value \den[\mdl{M},\gsn]{\sph} by recursively building up basic interpretation functions \fun{\mdl{I}}{.} as \fun{\mdl{I}}{\sph} for constants and the assignment function \fun{\gsn}{.} as \fun{\gsn}{x} for variables.

		\begin{enumerate}
			\item If \sph{} is a constant \f{\sph \in \fl\com \den[\mdl{M},\gsn]{\sph} = \fun{\mdl{I}}{\sph}}.
			
			\item If \sph{} is a variable in a set of variables \f{\sph \in V\com \den[\mdl{M},\gsn]{\sph} = \fun{\gsn}{\sph}}.
			
			\item For any \f{n}-ary predicate \f{P \in\fl} and sequence of terms \f{\sph_1,\sph_2,\dots,\sph_n}, we have that  \f{\den[\mdl{M},\gsn]{\fun{P}{\sph_1,\sph_2\dots,\sph_n}}=1\wenn}\begin{align*}
				\bp{\den[\mdl{M},\gsn]{\f{\sph_1}},\dots,\den[\mdl{M},\gsn]{\f{\sph_n}}}\in\den[\mdl{M},\gsn]{\f{P}}=\fun{\mdl{I}}{P}.
			\end{align*}
			
			\item For any \f{n}-ary function \f{f \in\fl} and sequence of terms \f{\sph_1,\sph_2,\dots,\sph_n}, we have that\begin{align*}
				\den[\mdl{M},\gsn]{\fun{f}{\sph_1,\sph_2,\dots,\sph_n}} & = \fun{\den[\mdl{M},\gsn]{\f{f}}}{\den[\mdl{M},\gsn]{\f{\sph_1}},\dots,\den[\mdl{M},\gsn]{\f{\sph_n}}} \\
				& =\fun{\fun{\mdl{I}}{f}}{\den[\mdl{M},\gsn]{\f{\sph_1}},\dots,\den[\mdl{M},\gsn]{\f{\sph_n}}}.
			\end{align*}
			
		\end{enumerate}
		
	\end{itemize}
\end{definition}

We may parse the technicalities of Definition~\ref{def:extmod}, and characterize the extensional model alternatively in prose-terms as:\begin{align*}
	\mdl[x]{M} = \bt{\esc{`}\text{\makecell{collection of type-specific \\ sets of allowed values}}\esc{'}, \esc{`}\text{\makecell{function that gives \\ meaning to symbols}}\esc{'}}.
\end{align*}

	Key properties of the formal system for semantics include: compositionality; context dependence; levels of meaning. The first, compositionality, has significant consequence for theories of semantics, not least of all for the primary concern of this paper, namely, showing that a relationship to a vector space model of semantics respects compositionality. Furthermore, as a consequence of this compositionlity, it is worthwhile to illustrate exactly how the denotation function is built recursively. 
	
Theorem~\ref{thm:recur} characterizes the recursiveness of denotation for an extensional semantics\footnote{For further reading on recursive functions and their interpretations, see foundational texts in \cite{SCOTT1993411,Scott1982,Scott1971}.}. The proof follows. However, to proceed, we need some measure of syntactic complexity such that, for any expression in a formal language, its denotation is well-defined by induction on its syntactic structure, where subexpressions are quantifiable as less complex than the whole expression containing them. We define such an expression following \cite{BLAIR198225,FABER2011278,Hodges2001HODFFO}.

We introduce a measure of complexity \plx{} that assigns a non-negative integer to each term and atomic formula, sensitive to the depth of its syntactic structure; that is, we define the complexity \fun{\plx}{s} of a structure \f{s} and subsequent structures \f{t_i} as the number of function and predicate symbols in \f{s}. In this way, the assigned complexity is 0 for constants and variables because they are atomic elements with no subcomponents; function terms and formulas have complexity one plus the maximum complexity of their immediate subterms, indicative of added layer(s) of structure. This measure reflects the depth of the syntactic structure of \f{s}, such that the complexity decreases with each recursive subcomponent.

We are probing tree structures. In this way, nodes represent function symbols, predicate symbols, constants, or variables; edges connect a parent node to its immediate subcomponents (arguments). Immediately, two important measures for such trees are: (1) depth, in which we check the length of the longest path from the root node to a leaf node; (2) size, the total number of nodes in the tree. We have some choice in the matter of calculating the complexity used in the proof of Theorem~\ref{thm:recur} for checking syntactic structure according to either depth or size.

If we prefer the \emph{depth} metric, then, for some structure of the formal language \f{s} and subsequent structures \f{t_i}:\begin{align*}
	\fun{\plx_{\text {depth }}}{s}= \begin{cases}0 & \text { if } s \text { is a constant or variable, } \\ 1+\max_{1 \leq i \leq n}\bc{\fun{\plx_{\text {depth }}}{t_i}} & \text { if } s=\fun{f}{t_1,t_2,\dots,t_n} \text { or } \fun{P}{t_1,t_2,\dots,t_n} .\end{cases}
\end{align*}

On the other hand, if we prefer the \emph{size} metric, then, for some element of the formal language \f{s}:\begin{align*}
	\fun{\plx_{\text {size }}}{s}= \begin{cases}0 & \text { if } s \text { is a constant or variable, } \\ 1+\sum_{i=1}^n \fun{\plx_{\text {size }}}{t_i} & \text { if } s=\fun{f}{t_1,t_2,\dots,t_n} \text { or } \fun{P}{t_1,t_2,\dots,t_n} .\end{cases}
\end{align*}

The depth metric \fun{\plx_{\text {depth }}}{s} measures the maximum path length from root to leaf in a term's tree representation, incrementing by 1 at each function application. For a term $s$, it equals the tree's depth minus 1 (when indexing at 0), exactly indicating the hierarchical structure of nested function applications. The size metric \fun{\plx_{\text {size }}}{s} measures the total number of nodes in a term's tree representation, where each symbol (function, predicate, constant, or variable) contributes 1 to the complexity. Through summation over all subterms, it accumulates these contributions, yielding the tree's total size.

\begin{example}
	Consider the term \fun{f}{a,b}, where \f{f} is a binary function symbol and \f{a} and \f{b} are constants. The syntax tree representation of this binary function has three nodes \bc{f,a,b} and depth given by: 1 (root \f{f}) \f{+} 1 (leaves \f{a, b}) \f{= 2} levels. The tree structure is:
	
	\begin{figure}[H]
		\centering
		\Tree[.f [.a ] [.b ]]
	\end{figure}
	
	Calculating \fun{\plx_{\text {depth }}}{s} follows: the base cases are given by:\begin{align*} & \funeq{\plx_{\text {depth }}}{a}{0} \\ 
		& \funeq{\plx_{\text {depth }}}{b}{0},
	\end{align*}
	
	while the function term is given by:\begin{align*}
		\fun{\plx_{\text {depth }}}{\fun{f}{a,b}}=1+\max_{1 \leq i \leq n}\bc{\fun{\plx_{\text {depth }}}{a},\fun{\plx_{\text {depth }}}{b}}=1+\max_{1 \leq i \leq n}\bc{0,0}=1+0=1.
	\end{align*}
	The depth calculated is 1, which corresponds to the depth from the root to the deepest leaf (excluding the root level).
	
	Contrast this now with the size metric. Calculating \fun{\plx_{\text {size }}}{s} follows: the base cases are given by:\begin{align*} & \funeq{\plx_{\text {size }}}{a}{1} \\ 
		& \funeq{\plx_{\text {size }}}{b}{1},
	\end{align*}
	
	while the function term is given by:\begin{align*}
		\fun{\plx_{\text {size }}}{\fun{f}{a,b}}=1+\fun{\plx_{\text {size }}}{a} + \fun{\plx_{\text {size }}}{b}=1+1+1=3.
	\end{align*}
	The size calculated is 3, which matches the total number of nodes in the tree.
\end{example}

\begin{example}
	Consider the nested term \fun{f}{\fun{g}{a},\fun{h}{b,c}}, where \f{f,g,h} are a binary function symbols and \f{a,b,c} are constants. The syntax tree representation of this function has six nodes \bc{f,g,h,a,b,c} and depth given by: 1 (root \f{f}) \f{+} max depth of subtrees \f{= 3} levels. The tree structure is:
	
	\begin{figure}[H]
		\centering
		\Tree[.f [.g [.a ]] [.h [.b ] [.c ]] ]
	\end{figure}
	
	Calculating \fun{\plx_{\text {depth }}}{s} follows: the base cases are given by:\begin{align*} & \funeq{\plx_{\text {depth }}}{a}{0} \\ 
		& \funeq{\plx_{\text {depth }}}{b}{0} \\
		& \funeq{\plx_{\text {depth }}}{c}{0},
	\end{align*}
	
	while the function terms are given by:\begin{align*}
		& \fun{\plx_{\text {depth }}}{\fun{g}{a}} = 1 + \fun{\plx_{\text {depth }}}{a} =1+0=1 \\
		& \fun{\plx_{\text {depth }}}{\fun{h}{b,c}} = 1 + \max_{1 \leq i \leq n}\bc{\fun{\plx_{\text {depth }}}{b},\fun{\plx_{\text {depth }}}{c}} =1+\max_{1 \leq i \leq n}\bc{0,0}=1+0=1\\
		& \fun{\plx_{\text {depth }}}{\fun{f}{\fun{g}{a},\fun{h}{b,c}}} = 1 + \max_{1 \leq i \leq n}\bc{\fun{\plx_{\text {depth }}}{\fun{g}{a}},\fun{\plx_{\text {depth }}}{\fun{h}{b,c}}} =1+\max_{1 \leq i \leq n}\bc{1,1}=1+1=2.
	\end{align*}
	
	The depth calculated is 2, which corresponds to the depth from the root to the deepest leaf.
	
	Contrast this now with the size metric. Calculating \fun{\plx_{\text {size }}}{s} follows: the base cases are given by:\begin{align*} & \funeq{\plx_{\text {size }}}{a}{1} \\ 
		& \funeq{\plx_{\text {size }}}{b}{1} \\
		& \funeq{\plx_{\text {size }}}{c}{1},
	\end{align*}
	
	while the function terms is given by:\begin{align*}
		& c_{\text {size }}(g(a))=1+c_{\text {size }}(a)=1+1=2 \\
		& c_{\text {size }}(h(b, c))=1+c_{\text {size }}(b)+c_{\text {size }}(c)=1+1+1=3\\
		& c_{\text {size }}(\fun{f}{\fun{g}{a},\fun{h}{b,c}})=1+c_{\text {size }}(g(a))+c_{\text {size }}(h(b, c))=1+2+3=6.
	\end{align*}

	The size calculated is 6, which is the total number of nodes in the tree.
\end{example}

\begin{example}
	Consider the nested term \fun{f}{\fun{g}{\fun{h}{a}}}, where \f{f,g,h} are unary function symbols and \f{a} is a constant. The syntax tree representation of this function has four nodes \bc{f,g,h,a} and depth given by: \f{1 (f) + 1 (g) + 1 (h) + 1 (a) = 4} levels. The tree structure is:
	
	\begin{figure}[H]
		\centering
		\Tree[.f [.g [.h [.a ]]]]
	\end{figure}
	
	Calculating \fun{\plx_{\text {depth }}}{s} follows: the base cases are given by:\begin{align*} & \funeq{\plx_{\text {depth }}}{a}{0}.
	\end{align*}
	
	while the function terms are given by:\begin{align*}
		& \fun{\plx_{\text {depth }}}{\fun{h}{a}} = 1 + \fun{\plx_{\text {depth }}}{a} =1+0=1 \\
		& \fun{\plx_{\text {depth }}}{\fun{g}{\fun{h}{a}}} = 1 + \max_{1 \leq i \leq n}\bc{\fun{\plx_{\text {depth }}}{\fun{h}{a}}} =1 +1=2\\
		& \fun{\plx_{\text {depth }}}{\fun{f}{\fun{g}{\fun{h}{a}}}} = 1 + \max_{1 \leq i \leq n}\bc{\fun{\plx_{\text {depth }}}{\fun{g}{\fun{h}{a}}}} =1 +2=3.
	\end{align*}
	
	The depth calculated is 3, which corresponds to the depth from the root to the deepest leaf.
	
	Contrast this now with the size metric. Calculating \fun{\plx_{\text {size }}}{s} follows: the base cases are given by:\begin{align*} & \funeq{\plx_{\text {size }}}{a}{1},
	\end{align*}
	
	while the function terms is given by:\begin{align*}
		& \fun{\plx_{\text {size }}}{\fun{h}{a}} = 1 + \fun{\plx_{\text {size }}}{a} = 1 +1 = 2 \\
		& \fun{\plx_{\text {size }}}{\fun{g}{\fun{h}{a}}} = 1 + \fun{\plx_{\text {size }}}{\fun{h}{a}} = 1 +2 = 3 \\
		& \fun{\plx_{\text {size }}}{\fun{f}{\fun{g}{\fun{h}{a}}}} = 1 + \fun{\plx_{\text {size }}}{\fun{g}{\fun{h}{a}}} = 1 +3 = 4.
	\end{align*}
	
	The size calculated is 4, which is the total number of nodes in the tree.
\end{example}

It follows, then, that the depth metric is preferred for when we want to measure relative to the depth or hierarchical structure of expressions, while the size metric for when we consider the total work or resource usage involved in processing the expression. Therefore, because we are interested in recursion-based nesting, and the depth metric essentially counts the number of function applications (excluding the base level), it stands to reason we proceed with the depth metric in the proofs for Theorem~\ref{thm:recur}.

\begin{theorem}\label{thm:recur}
	Let \fl{} be a formal language with constants, variables, function symbols, and predicate symbols. Let \f{\mathcal{T}} be a set of extensional types and \f{\mdl[x]{M} = \bt{\bp[{\tau\in\mathcal{T}}]{\mdl[\tau]{D}},\mdl{I}}} be an extensional model structure for \fl. Extensional denotation \den[\mdl{M},\gsn]{.} recursively builds semantic values for all terms and atomic formulas in the language \fl.

\end{theorem}

\begin{proof}\label{pr:recur}
	We proceed in three parts, establishing each property in turn: (1) \den[\mdl{M},\gsn]{s} is well-founded (meaning is built from atomic elements); (2) \den[\mdl{M},\gsn]{s} is compositional (meaning of a complex expression depends solely on the meanings of its immediate parts); (3) \den[\mdl{M},\gsn]{s} is fully determined (given the same inputs, it yields the same outputs).
	
	\begin{lemma}\label{lem:well}
		For every term or atomic formula \s{s} in \fl, the computation of \den[\mdl{M},\gsn]{s} terminates after a finite number of steps, starting from the atomic elements (constants and variables); the recursive process is well-founded and does not result in infinite regress.
	\end{lemma}
	
	\begin{proof}\label{pr:well}

		The proof follows by induction. For all terms and atomic formulas \f{s'} with complexity \f{\fun{\plx}{s'} < k}, the denotation \den[\mdl{M},\gsn]{s'} is well-defined .

		Consider the base case \funeq{\plx}{s}{0}. For constants and variables, the denotation is given directly by the base cases:\begin{align*}
			\den[\mdl{M},\gsn]{c} = \fun{\mdl{I}}{c} , \quad \den[\mdl{M},\gsn]{x} = \fun{\gsn}{x}.
		\end{align*}
		
		The recursion terminates immediately, and establishes the base case: for all \f{s} with \funeq{\plx}{s}{0}, \den[\mdl{M},\gsn]{s} is well-defined \sans any recursion.
		
		Now consider the inductive step \f{\fun{\plx}{s}=k>0}. Assume that the denotation is well-defined for all terms and atomic formulas \f{s'} with complexity \f{\fun{\plx}{s'}<k}  (induction hypothesis). 
		
		Let \f{s=\fun{f}{t_1,t_2,\dots,t_n}} be a function term. Then complexity follows as:\begin{align*}
			\fun{\plx}{s} = 1 + \max _{1 \leq i \leq n}\bc{\fun{\plx}{t_i}} = k.
		\end{align*}
		
		Therefore:\begin{align*}
			 \max _{1 \leq i \leq n}\bc{\fun{\plx}{t_i}} = k - 1,
		\end{align*}
		
		from which it follows that subterms have complexity: \f{\fun{\plx}{t_i}\leq k-1} for all \f{i=1,2,\dots,n}. By the induction hypothesis, \den[\mdl{M},\gsn]{t\textsubscript{i}} is well-defined for each \f{t_i}. Then, for functions, the denotation of \f{s} is given by:\begin{align*}
			\den[\mdl{M},\gsn]{s} = \fun{\fun{\mdl{I}}{f}}{\den[\mdl{M},\gsn]{t\textsubscript{1}},\den[\mdl{M},\gsn]{t\textsubscript{2}},\dots,\den[\mdl{M},\gsn]{t\textsubscript{n}}}.
		\end{align*}
		
		\fun{\mdl{I}}{f} is an \f{n}-ary function from \f{\mdl{D}^n} to \mdl{D}, and since each \f{\den[\mdl{M},\gsn]{t\textsubscript{i}}\in\mdl{D}}, the function application is well-defined. Therefore, \den[\mdl{M},\gsn]{s} is well-defined for \f{s=\fun{f}{t_1,t_2,\dots,t_n}}.
		
		Now let \f{s=\fun{P}{t_1,t_2,\dots,t_n}} be an atomic formula with:\begin{align*}
			\max _{1 \leq i \leq n}\bc{\fun{\plx}{t_i}} = k - 1,
		\end{align*}
		
		from which it follows that subterms have complexity: \f{\fun{\plx}{t_i}\leq k-1} for all \f{i=1,2,\dots,n}, similarly to function terms. Again, by the induction hypothesis, \den[\mdl{M},\gsn]{t\textsubscript{i}} is well-defined for each \f{t_i}; the denotation of \f{s} is given by:\begin{align*}
			\den[\mdl{M},\gsn]{s} = \begin{cases}1 & \text { if } \bp{\den[\mdl{M},\gsn]{t\textsubscript{1}},\den[\mdl{M},\gsn]{t\textsubscript{2}},\dots,\den[\mdl{M},\gsn]{t\textsubscript{n}}} \in \fun{\mdl{I}}{P}, \\ 0 & \text { otherwise. }\end{cases}.
		\end{align*}
		
		\fun{\mdl{I}}{P} is an \f{n}-ary relation over \mdl{D} (a subset of \f{\mdl{D}^n}), so the \f{n}-ary tuple \bp{\den[\mdl{M},\gsn]{t\textsubscript{1}},\den[\mdl{M},\gsn]{t\textsubscript{2}},\dots,\den[\mdl{M},\gsn]{t\textsubscript{n}}} is in that subset of \f{\mdl{D}^n}, from which checking membership in \fun{\mdl{I}}{P} is well-defined. Therefore, \mdl[\mdl{M},\gsn]{s} is well-defined for \f{s=\fun{P}{t_1,t_2,\dots,t_n}}.
		
		Since both cases show that \mdl[\mdl{M},\gsn]{s} is well-defined when the denotations of all subterms are well-defined, the inductive step is complete. the denotation function \den[\mdl{M},\gsn]{.} is well-defined for all terms and atomic formulas \f{s} in \fl. Therefore, the recursive definitions are well-founded, and the computation of \den[\mdl{M},\gsn]{s} terminates after a finite number of steps for every term and atomic formula.
		
	\end{proof}
	
	\begin{lemma}\label{lem:comp}
		For every term or atomic formula \s{s} in \fl, the denotation \den[\mdl{M},\gsn]{s} depends solely on:
		\begin{enumerate}
			\item the denotations of its immediate subcomponents \den[\mdl{M},\gsn]{t\textsubscript{i}} ;
			\item the interpretations \fun{\mdl{I}}{f} or \fun{\mdl{I}}{P} of the function or predicate symbols involved;
			\item the syntactic structure used to combine these subcomponents;
		\end{enumerate}
		
		such that the meaning of \f{s} is determined by the meanings of its parts and their syntactic arrangement.
	\end{lemma}
	
	\begin{proof}\label{pr:comp}
		The proof follows by induction.
		
		Consider the base case \funeq{\plx}{s}{0}. For constants and variables, the denotation is given directly by the base cases:\begin{align*}
			\den[\mdl{M},\gsn]{c} = \fun{\mdl{I}}{c} , \quad \den[\mdl{M},\gsn]{x} = \fun{\gsn}{x}.
		\end{align*}
		
		There are no subcomponents; thus, compositionality holds trivially. In both cases, the denotation is determined entirely by the interpretation function and involves no further components. Therefore, compositionality holds for \funeq{\plx}{s}{0}.

		Now consider the inductive step \f{\fun{\plx}{s}=k>0}. Assume that for all terms and atomic formulas \f{s'} with complexity \f{\fun{\plx}{s'}<k}, the denotation\den[\mdl{M},\gsn]{s'} depends solely on  (induction hypothesis): (1) the denotations of its immediate subcomponents; (2) the interpretations of the function or predicate symbols involved; (3) the syntactic rules used to combine them.
		
		Let \f{s=\fun{f}{t_1,t_2,\dots,t_n}} be a function term with \funeq{\plx}{s}{k}. Each subterm has complexity \f{\fun{\plx}{t_i}\leq k-1} for all \f{i=1,2,\dots,n}; by the induction hypothesis, the denotations \den[\mdl{M},\gsn]{t\textsubscript{i}} are compositional. The denotation of \f{s} depends on: the interpretation \fun{\mdl{I}}{f} of the function symbol \f{f} and the denotations \den[\mdl{M},\gsn]{t\textsubscript{i}} of its immediate subterms; the syntactic rule of function application\footnote{ If $\alpha$ is a branching node with two daughter constituents, \bc{\beta,\gamma}, \f{\den[\mdl{M},\gsn]{\f{\gamma}}\in\dom{\den[\mdl{M},\gsn]{\f{\beta}}}}, then \f{\den[\mdl{M},\gsn]{\f{\alpha}} = \fun{\den[\mdl{M},\gsn]{\f{\beta}}}{\den[\mdl{M},\gsn]{\f{\gamma}}}}.} combines these denotations. Since \den[\mdl{M},\gsn]{s} is computed directly from the denotations of \f{t_1,t_2,\dots,t_n} and \fun{\mdl{I}}{f}, the denotation of \f{s} depends solely on its immediate subcomponents and their interpretations by the syntactic structure (function application) used to combine them.
		
		Similarly, let \f{s=\fun{P}{t_1,t_2,\dots,t_n}} be an atomic formula with term complexity likewise \f{\fun{\plx}{t_i}\leq k-1}. Again, by the induction hypothesis, the denotations \den[\mdl{M},\gsn]{t\textsubscript{i}} are compositional. The truth of \f{s} depends on the interpretation of the predicate symbol \fun{\mdl{I}}{P} and the denotations of immediate subterms, while the syntactic rule of predicate evaluation\footnote{If $\alpha$ is a branching node with two daughter constituents, \bc{\beta,\gamma}, then, if \f{\den[\mdl{M},\gsn]{\f{\beta}}\in\mdl[\bt{e,t}]{D}} and \f{\den[\mdl{M},\gsn]{\f{\gamma}}\in\mdl[\bt{e,t}]{D}}, then \f{\den[\mdl{M},\gsn]{\f{\alpha}} = \bq{\lambda x \in\mdl[e]{D}.\funeq{\den[\mdl{M},\gsn]{\f{\beta}}}{x}{1}\et\funeq{\den[\mdl{M},\gsn]{\f{\gamma}}}{x}{1} }}}. checks membership in \fun{\mdl{I}}{P}. Now again, since \den[\mdl{M},\gsn]{s} is determined entirely by the denotations of the terms \f{t_1,t_2,\dots,t_n} and whether they satisfy the predicate \f{P}, the denotation of $s$ depends solely on its immediate subcomponents and their interpretations via predicate application.
		
		Therefore, if the denotations of the immediate subcomponents of \f{s} are compositional (by inductive hypothesis), then the denotation of \f{s} is also compositional; by induction on the complexity \fun{\plx}{s}, the denotation function \den[\mdl{M},\gsn]{.} is compositional for all terms and atomic formulas \f{s}.
		
	\end{proof}
	
	\begin{lemma}\label{lem:deter}
		Given the same model structure \mdl{M}, variable assignment \gsn, and expression \f{s}, the denotation \den[\mdl{M},\gsn]{s} is uniquely determined.
	\end{lemma}
	
	\begin{proof}\label{pr:deter}
		As before, the proof follows by induction.

		Consider the base case \funeq{\plx}{s}{0}. For constants and variables, the denotation is given directly by the base cases:\begin{align*}
			\den[\mdl{M},\gsn]{c} = \fun{\mdl{I}}{c} , \quad \den[\mdl{M},\gsn]{x} = \fun{\gsn}{x}.
		\end{align*}
		
		\fun{\mdl{I}}{c} and \fun{\gsn}{c} are uniquely determined by \mdl{I} and \gsn, respectively, trivially, which are fixed. For all \f{s} with\funeq{\plx}{s}{0}, the denotation \den[\mdl{M},\gsn]{s} is uniquely determined. Therefore, the denotation function is deterministic for base cases.

		Now for the inductive step \f{\fun{\plx}{s}=k>0}. Assume that for all terms and atomic formulas \f{s'} with complexity \f{\fun{\plx}{s'}<k}, the denotation\den[\mdl{M},\gsn]{s'} is uniquely determined (induction hypothesis).

		Let \f{s=\fun{f}{t_1,t_2,\dots,t_n}} be a function term with \funeq{\plx}{s}{k}. Each subterm has complexity \f{\fun{\plx}{t_i}\leq k-1} for all \f{i=1,2,\dots,n}; by the induction hypothesis, the denotations \den[\mdl{M},\gsn]{t\textsubscript{i}} are uniquely determined. The denotations of subterms \den[\mdl{M},\gsn]{t\textsubscript{i}} are uniquely determined by the induction hypothesis, while \fun{\mdl{I}}{f} is fixed and uniquely defined. Since \fun{\mdl{I}}{f} is a function by definition of interpretation, applying it to these uniquely determined arguments yields a unique result, thus establishing that \den[\mdl{M},\gsn]{s} is uniquely determined.

		Let \f{s=\fun{P}{t_1,t_2,\dots,t_n}} be an atomic formula with term complexity likewise \f{\fun{\plx}{t_i}\leq k-1}. Again, by the induction hypothesis, the denotations \den[\mdl{M},\gsn]{t\textsubscript{i}} are uniquely determined. Each \den[\mdl{M},\gsn]{t\textsubscript{i}} is uniquely determined, and \fun{\mdl{I}}{P} is a fixed subset of \f{\mdl{D}^n}; since the tuple \bp{\den[\mdl{M},\gsn]{t\textsubscript{1}},\den[\mdl{M},\gsn]{t\textsubscript{2}},\dots,\den[\mdl{M},\gsn]{t\textsubscript{n}}} is uniquely determined, and checking its membership in \fun{\mdl{I}}{P} is a well-defined, deterministic operation, it follows that \den[\mdl{M},\gsn]{s} is uniquely determined.
		
		In both cases, we have shown that if the denotations of the immediate subcomponents of \f{s} are uniquely determined (by the induction hypothesis), then the denotation of \f{s} is also uniquely determined. Therefore, by induction on the complexity \fun{\plx}{s}, the denotation function \den[\mdl{M},\gsn]{.} is deterministic for all terms and atomic formulas \f{s}.

	\end{proof}

	Therefore, the denotation function \den[\mdl{M},\gsn]{.} recursively constructs semantic values from atomic elements (well-foundedness), while each complex term or atomic formula's meaning is determined exclusively by its immediate subcomponents' meanings and their syntactic arrangement (compositionality). Moreover, given identical inputs of the model, variable assignment, and expression, the denotation function consistently produces the same output (determinism).

\end{proof}

The denotation function begins with basic interpretations \mdl{I} of constants, function and predicate symbols, along with variable assignment \gsn. Through recursive application of definitions, it computes denotations for complex terms and atomic formulas, ultimately yielding a final semantic value that adheres to well-foundedness, compositionality, and determinism.\begin{align*}
		\text { Basic interpretations } \mathcal{I}, g \quad \xrightarrow{\text { Recursive application }} \quad \text { Final semantic value }
\end{align*}

We belabor the point on the interpretation and denotation recursions, for it is this machinery that is responsible for guiding the compositionality discussion in this paper.

\section{Distributional semantics}\label{sec:vs}
The distributional approach to semantics represents the meaning of words as distributions in a high-dimensional vector space. From the outset, a vector space model of semantics was interdisciplinary: linguistics, psychology, and computer science, each of which contributed a fundamental aspect of the approach. The linguistic contribution followed from the distributionalists above mentioned, though the conceit of representing linguistic information in the form of vector-like arrays of information has roots at least as far back as Descartes and Leibniz \cite{wierzbickasemantics1992,wierzbickasemantics1996}, developed in the work of \cite{hjelmsleprolegomena1961}, and fleshed out in early models of generative grammar \cite{katz1963}.

Vector spaces and techniques from linear algebra are ubiquitous in applied mathematics and physics, but they also occur in areas such as cognitive science \cite{gard2000}, machine learning \cite{bishoppattern2006,kelleherfundamentals2020}, computational linguistics \cite{voav2015,grefenstetteexperimental2011,maillardtypedriven2014,amblardcontext2016}, the social sciences \cite{baltagdynamic2019,liulogical2014} and formal philosophy \cite{vanb2023,leit2020}. There is also a body of logical work on vector spaces, in the first-order model-theoretic tradition \cite{hodges1993,pill1997}, in relevant logic \cite{urq1972} and in modal logics of space \cite{vanbenthemmodal2007}. Distributional models reliably correlate many linguistic processes, including synonymy, noun categorization, selectional preference, analogy, and relation classification \cite{landauer1997solution,lund1995semantic,baroni2010distributional,erk2010flexible}. Psychologically, there is evidence that brain activity and the arrangement of hierarchical information in the brain is not only distributional, but may be modeled as a specific kind of vector space (hyperbolic space), a consequence of abstraction over context of use, experience, and and processing of hierarchical information \cite{mitchell2008a,mitchell2008b,zhang2023hippocampal}.

\subsection{Vector model of semantics}

The mathematics that underlie distributional semantics is that of vector spaces; the meaning of words, phrases, and sentences are vectors in a high-dimensional space. The underlying idea is that the meaning of a word can be captured by its contextual usage, with similar contexts leading to similar meanings.

\begin{definition}[Vector space]\label{def:vectorspace}
	\f{\mdl{V} = \bt{V,\oplus,\odot}} is a \term{vector space} over a field \ff{K} (conventionally taken to be \ff{R} in natural language processing) with two binary operations:

	\begin{enumerate}
		\item \term{vector addition} \f{\oplus: V \times V \rightarrow V} that satisfies:
		\begin{enumerate}
			\item \f{\vv{u}\oplus\vv{v} = \vv{v}\oplus\vv{u}} for all \f{\vv{u},\vv{v}\in V} \hfill commutativity
			\item \f{\bp{\vv{u}\oplus\vv{v}}\oplus\vv{w} = \vv{u}\oplus\bp{\vv{u}\oplus\vv{w}}} for all \f{\vv{u},\vv{v},\vv{w}\in V} \hfill associativity
			\item \f{\exists \vv{0}\in V} such that \f{\vv{v}\oplus \vv{0}=\vv{v}} for all \f{\vv{v}\in V} \hfill zero vector
			\item \f{\forall \vv{v}\in V}, there exists a vector \f{-\vv{v}\in V} such that \f{\vv{v}\oplus\bp{-\vv{v}}=\vv{0}} \hfill additive inverse
		\end{enumerate}
		\item \term{scalar multiplication} \f{\odot: \ff{K} \times V \rightarrow V} that satisfies:
		\begin{enumerate}
			\item \f{\alpha\odot\bp{\vv{v}\oplus\vv{u}}=\alpha\odot\vv{v}\oplus\alpha\odot\vv{u}} for all \f{\alpha\in\ff{K}} and \f{\vv{u},\vv{v}\in V} \hfill distributive over vector addition
			\item \f{\bp{\alpha\oplus\beta}\odot\vv{v} = \alpha\odot\vv{v}\oplus\beta\odot\vv{v}} \hfill distributive over field addition
			\item \f{\alpha\odot\bp{\beta\odot\vv{v}} = \bp{\alpha\beta}\odot\vv{v}} for all \f{\alpha,\beta\in\ff{K}} and \f{\vv{v}\in V} \hfill associativity
			\item \f{1\odot\vv{v} = \vv{v}} for all \f{\vv{v}\in V} \hfill identity element
		\end{enumerate}
	\end{enumerate}

\end{definition}

\begin{definition}[Inner product space]\label{def:innerprod}
	A vector space \mdl{V} equipped with an inner product is called an \term{inner product space}. The inner product is a function \f{\inn{\cdot}{\cdot}:V\times V\rightarrow \ff{K}} that satisfies:
	 
	\begin{enumerate}
		\item \f{\inn{\vv{v}}{\vv{u}}=\overline{\inn{\vv{u}}{\vv{v}}}} \hfill conjugate symmatry
		\item \f{\inn{\alpha\vv{u}\oplus\beta\vv{v}}{\vv{w}} = \alpha\inn{\vv{v}}{\vv{u}}\oplus\beta\inn{\vv{v}}{\vv{w}}} \hfill linearity
		\item \f{\inn{\vv{v}}{\vv{v}}\geq0} with equality if and only if \f{\vv{v}=\vv{0}} \hfill positivity
	\end{enumerate}
\end{definition}

\begin{definition}[Hilbert space]\label{def:hilbert}
	A \term{Hilbert space} \f{\mdl{S} = \bt{\mdl{V},\inn{\cdot}{\cdot}}} is a real or complex inner product space that is complete with respect to the norm induced by the inner product, where:
	
	\begin{itemize}
		\item \mdl{V} is a (possibly infinite-dimensional) vector space,
		\item \inn{\cdot}{\cdot} is an inner product (positive-definite sesquilinear form),
	\end{itemize}

	and is complete: every Cauchy sequence in \mdl{V} converges to an element in \mdl{V}.
\end{definition}

A note on notation. We are choosing to use \inn{\cdot}{\cdot} to represent the inner product instead of the more common \bt{.~,~.}, so as to avoid possible confusion with the notation we are using for model \f{n}-tuples.

\begin{definition}[Metric]\label{def:metric}
	The \term{norm} of a vector \f{\vv{v}\in V} is defined as: \begin{align*}
		\norm{\vv{v}} = \sqrt{\inn{\vv{v}}{\vv{v}}}.
	\end{align*}
	
	The \term{metric} is a distance function: \begin{align*}
	\funeq{d}{\vv{v},\vv{u}}{\norm{\vv{u}-\vv{v}}}.
	\end{align*}
\end{definition}

\begin{definition}[Semantic space]\label{def:semanticspace}
	A \term{semantic space} \mdl{S} is a Hilbert space, a complete inner product space; vectors represent semantic entities, and the operations of vector addition and scalar multiplication allow for the combination and scaling of these entities within the space.
	
\end{definition}

\begin{definition}[Word vector]\label{def:wordvector}
	Given a basis \f{\bc[{i=1}]{\vv{b}}^n} of linearly independent vectors, a \term{word vector} \f{\vv[w]{v}\in \mdl{S}} is the vector representation of a word \f{w} in the space \mdl{S} expressed as:\begin{align*}
		\vv[w]{v} =\sum_{i=1}^n \alpha_i^{\bp{w}} \vv[i]{b},
	\end{align*}
	
	where \f{\alpha_i^{\bp{w}}\in\ff{K}} is the corresponding coefficient in the field \ff{K}.
\end{definition}

The word vector, then, is the weighted sum of basic meaning features (following \cite{baronietal2014frege}, it is a summary of the contexts in which the word can occur):\begin{align*}
	\vv[w]{v} = \sum \text{(importance of feature}_i) \times \text{(feature dimension}_i)
\end{align*}

Some example word vectors may have the form, for some weights \f{\alpha_i,\beta_i}:\begin{align*}
	\mathbf{v}_\text{cat} &= \alpha_1\vv[furry]{b} + \alpha_2\vv[pet]{b} + \alpha_3\vv[keyboard]{b} + \dots \\
	\mathbf{v}_\text{run} &= \beta_1\vv[motion]{b} + \beta_2\vv[speed]{b} + \beta_3\vv[exercise]{b} + \dots
\end{align*}

Word vectors \vv[w]{v} are typically learned from large corpora using techniques such as co-occurrence matrices \cite{li2015word,salah2018word,bullinaria2012extracting}, neural networks \cite{mikolov2013efficientestimationwordrepresentations,penningtonetal2014glove,Mikolov2013,Bojanowski2017}, or matrix factorization methods like singular value decomposition (SVD) \cite{bullinaria2012extracting,albright2004taming,arora2012learning}. \cite{lupyan2024} provides a model whereby word vectors are encoded (though not explicitly) according to their type (\eg, as defined in Definition~\ref{def:exttyp}). These methods capture the context-dependent meaning of words by embedding them in a space where semantic similarity corresponds to vector proximity.

\begin{definition}[Vector model]\label{def:vectormodel}
	A \term{vector model} of semantics \mdl[\mdl{S}]{M} is simply the semantic space \mdl{S}.
	
\end{definition}

As we did for the extensional model, we may parse the technicalities alternatively in prose-terms. The Hilbert space, for example:\begin{align*}
	\mdl{S} = \bt{\esc{`}\text{vector space}\esc{'},\esc{`}\text{\makecell{way of measuring similarity \\ similarity between vectors}}\esc{'}},
\end{align*}

where that way of measuring similarities must respect vector operations, give positive self-similarity, and complete all limit points. The semantic space, then, has the following prose-form:\begin{align*}
	\mdl{S} = \bt{\esc{`}\text{\makecell{space of meaning \\ and word vectors}}\esc{'},\esc{`}\text{\makecell{way to measure and \\ compare meanings}}\esc{'}}.
\end{align*}

The semantic model \mdl[{\mdl{S}}]{M} is this inner product space on meaning vectors.

\begin{definition}[Similarity]\label{def:similarity}
	The \term{similarity} between two vectors \f{\vv{u},\vv{v}\in S} is often measured by the cosine of the angle \f{\theta} between them, given by:\begin{align*}
		\fun{\sml}{\vv{u},\vv{v}}=\funeq{\cos}{\theta_{\vv{u},\vv{v}}}{\frac{\inn{\vv{u}}{\vv{v}}}{\norm{\vv{u}} \norm{\vv{v}}}},
	\end{align*}
	
	where \norm{\vv{u}} is the norm of \vv{u} and similarly for \vv{v}.
	
\end{definition}

The cosine similarity ranges \bq{-1,+1}: \f{+1} indicates identical vectors (\ie, pointing in the same direction); \f{0} indicates orthogonality; \f{-1} indicates vectors pointing in opposite directions. Raw frequency values are non-negative; therefore, in practice, the similarity metric varies \bq{0,+1}. This particular measure is widely used because it is invariant to the magnitude of the vectors and thus captures the similarity based purely on direction in the semantic space \cite{jurafskyspeech2009}.

\section{Homomorphism}\label{sec:homo}
For any semantic relationship or operation in \mdl[ext]{M}, there exists a corresponding geometric relationship or operation in \mdl[{\mdl{S}}]{M} that reflects this semantic content; function application and composition in \mdl[ext]{M} correspond to the composition of functions in \mdl[\mdl{S}]{M}.

\begin{theorem}\label{thm:homo}
	For the extensional model \mdl[ext]{M} and the vector space model \mdl[{\mdl{S}}]{M}, there exist injective mappings \f{H = \bc[\tau\in T]{\mph[\tau]{h}}} where each \f{\mph[\tau]{h}:\mdl[\tau]{D}\rightarrow \img{\mph[\tau]{h}}\subseteq\mdl[{\mdl[\tau]{D}}]{S}} for each semantic domain \mdl[\tau]{D} in \mdl[ext]{D}, where \mdl[{\mdl[\tau]{D}}]{S} is a vector space in \mdl[\mdl{S}]{M}. For every semantic function \f{f:\mdl[A]{D}\rightarrow\mdl[B]{D}} in \mdl[ext]{M}, there exists a corresponding function \f{\mph[f]{h}:\img{\mph[A]{h}}\rightarrow\img{\mph[B]{h}}} in \mdl[\mdl{S}]{M} such that for all \f{a\in\mdl[A]{D}}:\begin{align*}
		\funeq{\mph[f]{h}}{\fun{\mph[A]{h}}{a}}{\fun{\mph[B]{h}}{\fun{f}{a}}}.
	\end{align*}

\end{theorem}

Moreover, the following commutative diagram holds:

\begin{figure}[H]
	\[\begin{tikzcd}
		\mdl[A]{D} && \mdl[B]{D} \\
		\\
		\img{\mph[A]{h}} && \img{\mph[B]{h}}
		\arrow["{\mph{f}}", from=1-1, to=1-3]
		\arrow["{\mph[A]{h}}"', from=1-1, to=3-1]
		\arrow["{\mph[B]{h}}", from=1-3, to=3-3]
		\arrow["{\mph[f]{h}}"', from=3-1, to=3-3]
	\end{tikzcd}\]
\end{figure}

A note on the injectivity of \mph[\tau]{h} with \mdl[{\mdl[\tau]{D}}]{S} and the surjectivity with \img{\mph[\tau]{h}}. To begin, we cannot claim a surjection with of \mph[\tau]{h} with \mdl[{\mdl[\tau]{D}}]{S}, as a consequence of working with infinite sets. If \mdl[e]{D} is infinite, then mapping it bijectively into a finite-dimensional vector space of the form, for example, \f{\mdl{S} = \ff[n]{R}} is not possible. The cardinality of \ff[n]{R} for finite \f{n} is equal to the cardinality of \ff{R}, which is the cardinality of the continuum \f{\mathfrak{c}}. The cardinality of infinite sets of entities (\eg, the set of natural numbers \ff{N}) can be countably infinite (\f{\aleph_0}) or uncountably infinite (\f{\mathfrak{c}}). It follows, then, that if \mdl[e]{D} has cardinality greater than \f{\mathfrak{c}}, a bijection cannot exists. So, unless we impose that \mdl[{\mdl[\tau]{D}}]{S} is infinite-dimensional or \mdl[e]{D} is finite, a bijective \mph[e]{h} cannot exist.

\subsection{Proof for \mph[e]{h}}

Let \f{\mdl[e]{D} = \bcdef{e_i}{i\in I}} be the set of entities in the extensional model \mdl[ext]{M}, where \f{I} is an index set. Let \mdl[{\mdl[e]{D}}]{S} be a vector space over \ff{R} with a basis \bcdef{\vv[i]{b}}{i\in I}, where each \vv[i]{b} is a vector in \mdl[{\mdl[e]{D}}]{S}. Define the mapping \f{\mph[e]{h}:\mdl[e]{D}\rightarrow\img{\mph[e]{h}}\subseteq\mdl[{\mdl[e]{D}}]{S}} by\begin{align*}
	\funeq{\mph[e]{h}}{e_i}{\vv[i]{b}}
\end{align*} that preserves the structural and operational properties of the extensional model \mdl[ext]{M} in the vector space model \mdl[{\mdl[e]{D}}]{S}.

\begin{lemma}\label{lem:bij}
	The mapping \f{\mph[e]{h}:\mdl[e]{D}\rightarrow\img{\mph[e]{h}}\subseteq\mdl[{\mdl[e]{D}}]{S}} is an injective mapping between \mdl[e]{D} and \mdl[{\mdl[e]{D}}]{S}.

\end{lemma}

\begin{proof}
	Suppose \f{e_i,e_j\in\mdl[e]{D}} and \f{e_i\neq e_j}. Then, by definition, \funeq{\mph[e]{h}}{e_i}{\vv[i]{v}} and \funeq{\mph[e]{h}}{e_j}{\vv[j]{v}}. Since basis vectors are linearly independent and distinct, we have that \f{\vv[i]{v}\neq \vv[j]{v}} for indices \f{i\neq j}, from which follows that \f{\fun{\mph[e]{h}}{e_i}\neq\fun{\mph[e]{h}}{e_j}}, and \mph[e]{h} is injective.
	
\end{proof}

\begin{lemma}\label{lem:rel}
	
	Let \f{\mdl{R}\subseteq\mdl[e]{D}^n} be an \f{n}-ary relation in the extensional model \mdl[ext]{M}. Then there exists a corresponding relation \f{\mdl{R}'\subseteq\mdl[{\mdl[e]{D}}]{S}^n} in the vector space model \mdl[{\mdl{S}}]{M} such that for all \f{e_1,e_2,\dots,e_n \in \mdl[e]{D}}:\begin{align*}
		\fun{\mdl{R}}{e_1,e_2,\dots,e_n}\Longleftrightarrow \fun{\mdl{R}'}{\fun{\mph[e]{h}}{e_1},\fun{\mph[e]{h}}{e_2},\dots,\fun{\mph[e]{h}}{e_n}},
	\end{align*}

	where \f{\mph[e]{h}:\mdl[e]{D}\rightarrow \img{\mph[e]{h}}\subseteq\mdl[{\mdl[e]{D}}]{S}} is an injective mapping from the entity domain to a subset of a vector space.

\end{lemma}

\begin{proof}

	Note that we define \f{\mdl{R}'} on \f{\mdl[{\mdl[e]{D}}]{S}^n}, the \f{n}-fold Cartesian product of \f{\mdl[{\mdl[e]{D}}]{S}}. From Lemma~\ref{lem:bij}, \mph[e]{h} is injective, and so has a left-inverse:\begin{align*}
		\funeq{\mph[e]{h}^{-1}}{\fun{\mph[e]{h}}{e_i}}{e_i},\quad \forall i=1,2, \dots, n .
	\end{align*}
	
	Then, for all \f{\vv[1]{v},\vv[2]{v},\dots,\vv[n]{v}\in \f{\mdl[{\mdl[e]{D}}]{S}}}:\begin{align*}
		\fun{\mdl{R}'}{\vv[1]{v},\vv[2]{v},\dots,\vv[n]{v}}= \begin{cases}\fun{\mdl{R}}{\fun{\mph[e]{h}^{-1}}{\vv[1]{v}},\fun{\mph[e]{h}^{-1}}{\vv[2]{v}},\dots,\fun{\mph[e]{h}^{-1}}{\vv[n]{v}}}, & \text { if } \vv[i]{v} \in \img{\mph[e]{h}^{-1}} \text { for all } i \\ 0, & \text { otherwise. }\end{cases}
	\end{align*}

	Let \f{e_1,e_2,\dots,e_n\in \mdl[e]{D}}. Suppose, then, that we have that \fun{\mdl{R}}{e_1,e_2,\dots,e_n} holds in \mdl[ext]{M}. Consider \fun{\mdl{R}'}{\fun{\mph[e]{h}}{e_1},\fun{\mph[e]{h}}{e_2},\dots,\fun{\mph[e]{h}}{e_n}}. We have, by definition of \f{\mdl{R}'}:\begin{align*}
		\fun{\mdl{R}'}{\fun{\mph[e]{h}}{e_1},\fun{\mph[e]{h}}{e_2},\dots,\fun{\mph[e]{h}}{e_n}} = \fun{\mdl{R}}{\fun{\mph[e]{h}^{-1}}{\fun{\mph[e]{h}}{e_1}},\fun{\mph[e]{h}^{-1}}{\fun{\mph[e]{h}}{e_2}},\dots,\fun{\mph[e]{h}^{-1}}{\fun{\mph[e]{h}}{e_n}}}.
	\end{align*}
	
	Since \funeq{\mph[e]{h}^{-1}}{\fun{\mph[e]{h}}{e_i}}{e_i} for all \f{i}, this simplifies to:\begin{align*}
		\begin{aligned}
			\fun{\mdl{R}'}{\fun{\mph[e]{h}}{e_1},\fun{\mph[e]{h}}{e_2},\dots,\fun{\mph[e]{h}}{e_n}} & =\fun{\mdl{R}}{\fun{\mph[e]{h}^{-1}}{\fun{\mph[e]{h}}{e_1}},\fun{\mph[e]{h}^{-1}}{\fun{\mph[e]{h}}{e_2}},\dots,\fun{\mph[e]{h}^{-1}}{\fun{\mph[e]{h}}{e_n}}} \\
			& =\fun{\mdl{R}}{e_1,e_2,\dots,e_n} .
		\end{aligned}
	\end{align*}
	
	Now, given that \fun{\mdl{R}}{e_1,e_2,\dots,e_n} is true by our initial assumption that it holds in \mdl[ext]{M}, it follows that \fun{\mdl{R}'}{\fun{\mph[e]{h}}{e_1},\fun{\mph[e]{h}}{e_2},\dots,\fun{\mph[e]{h}}{e_n}} is true in \mdl[{\mdl[e]{D}}]{S}.
	
	Conversely, suppose \fun{\mdl{R}'}{\fun{\mph[e]{h}}{e_1},\fun{\mph[e]{h}}{e_2},\dots,\fun{\mph[e]{h}}{e_n}} holds in \mdl[{\mdl{S}}]{M}. By the definition of \f{\mdl{R}'}:\begin{align*}
		\fun{\mdl{R}'}{\fun{\mph[e]{h}}{e_1},\fun{\mph[e]{h}}{e_2},\dots,\fun{\mph[e]{h}}{e_n}} = \fun{\mdl{R}}{\fun{\mph[e]{h}^{-1}}{\fun{\mph[e]{h}}{e_1}},\fun{\mph[e]{h}^{-1}}{\fun{\mph[e]{h}}{e_2}},\dots,\fun{\mph[e]{h}^{-1}}{\fun{\mph[e]{h}}{e_n}}}.
	\end{align*}
	
	Again, since \funeq{\mph[e]{h}^{-1}}{\fun{\mph[e]{h}}{e_i}}{e_i} for all \f{i}, we have:\begin{align*}
		\fun{\mdl{R}'}{\fun{\mph[e]{h}}{e_1},\fun{\mph[e]{h}}{e_2},\dots,\fun{\mph[e]{h}}{e_n}} = \fun{\mdl{R}}{e_1,e_2,\dots,e_n}.
	\end{align*}
	
	Given that \fun{\mdl{R}'}{\fun{\mph[e]{h}}{e_1},\fun{\mph[e]{h}}{e_2},\dots,\fun{\mph[e]{h}}{e_n}} is true, it follows that \fun{\mdl{R}}{e_1,e_2,\dots,e_n} is true in \mdl[ext]{M}.
	
	For all \f{e_1,e_2,\dots,e_n\in \mdl[e]{D}}:\begin{align*}
		\fun{\mdl{R}}{e_1,e_2,\dots,e_n}\Longleftrightarrow \fun{\mdl{R}'}{\fun{\mph[e]{h}}{e_1},\fun{\mph[e]{h}}{e_2},\dots,\fun{\mph[e]{h}}{e_n}}.
	\end{align*}
	
	Therefore, the mapping \mph[e]{h} preserves the \f{n}-ary relation \mdl{R} from \mdl[ext]{M} to \mdl[{\mdl{S}}]{M}.

\end{proof}

The truth of the relation among entities in \mdl[ext]{M} is exactly reflected by the truth of the relation among their vector representations in \mdl[{\mdl{S}}]{M}, which, in turn, holds for any \f{n}-ary relation (\ie, unary, binary, or higher arity).

\begin{corollary}
	Let \f{\mdl{R}\subseteq \mdl[e]{D}} be a unary relation. Then \f{\mdl{R}'\subseteq\mdl[{\mdl[e]{D}}]{S}} is defined by:\begin{align*}
		\fun{\mdl{R}'}{\vv{v}} = \fun{\mdl{R}}{\fun{\mph[e]{h}^{-1}}{\vv{v}}}.
	\end{align*}
	For \f{e\in\mdl[e]{D}}:\begin{align*}
		\fun{\mdl{R}}{e} = \fun{\mdl{R}'}{\fun{\mph[e]{h}}{e}}.
	\end{align*}
\end{corollary}

\begin{corollary}
	Let \f{\mdl{R}\subseteq \mdl[e]{D}\times \mdl[e]{D}} be a binary relation. Then \f{\mdl{R}'\subseteq\mdl[{\mdl[e]{D}}]{S}\times \mdl[{\mdl[e]{D}}]{S}} is defined by:\begin{align*}
		\fun{\mdl{R}'}{\vv[1]{v},\vv[2]{v}} = \fun{\mdl{R}}{\fun{\mph[e]{h}^{-1}}{\vv[1]{v}},\fun{\mph[e]{h}^{-1}}{\vv[2]{v}}}.
	\end{align*}
	For entities \f{e_1,e_2\in\mdl[e]{D}}:\begin{align*}
		\fun{\mdl{R}}{e_1,e_2} = \fun{\mdl{R}'}{\fun{\mph[e]{h}}{e_1},\fun{\mph[e]{h}}{e_2}}.
	\end{align*}
\end{corollary}

Set-membership is emphatically foundational in formal semantics; it underlies the interpretation of linguistic expressions as sets, enables predication through membership relations, forms the basis for quantification, and structures the type theory. The semantic denotation of \ol{cat}, for example, is interpreted as the set of all cats. It follows, then, to check respect of set-membership by \mph[e]{h}. 

\begin{corollary}\label{lem:set}
	Let \f{E\subseteq\mdl[e]{D}} be a subset in the extensional model \mdl[ext]{M}. There exists a corresponding subset \f{E'\subseteq\mdl[{\mdl[e]{D}}]{S}} in the vector space model \mdl[{\mdl{S}}]{M} such that for all \f{e\in\mdl[e]{D}}:\begin{align*}
		e\in E \Longleftrightarrow \fun{\mph[e]{h}}{e}\in E'.
	\end{align*}
\end{corollary}

Semantic constructs based on set-membership (predicates, quantifiers) remain otherwise consistent when represented in the vector space model \mdl[{\mdl{S}}]{M}; any subset \f{E} of \mdl[e]{D} in \mdl[ext]{M} corresponds precisely to \f{E'} of \mdl[{\mdl[e]{D}}]{S} in \mdl[{\mdl{S}}]{M}. For example, union, intersection, difference, and power set are preserved under \mph[e]{h}.

\begin{corollary}
	Union preservation under \mph[e]{h} is given by:\begin{align*}
		\funeq{\mph[e]{h}}{E_1\cup E_2}{\fun{\mph[e]{h}}{E_1}\cup\fun{\mph[e]{h}}{E_2}}.
	\end{align*}
\end{corollary}

\begin{corollary}
	Intersection preservation under \mph[e]{h} is given by:\begin{align*}
		\funeq{\mph[e]{h}}{E_1\cap E_2}{\fun{\mph[e]{h}}{E_1}\cap\fun{\mph[e]{h}}{E_2}}.
	\end{align*}
\end{corollary}

\begin{corollary}
	Set difference preservation under \mph[e]{h} is given by:\begin{align*}
		\funeq{\mph[e]{h}}{E_1\setminus E_2}{\fun{\mph[e]{h}}{E_1}\setminus\fun{\mph[e]{h}}{E_2}}.
	\end{align*}
\end{corollary}

\begin{corollary}
	The power set \fun{\mdl{P}}{\mdl[e]{D}} corresponds to the power set \fun{\mdl{P}}{\mdl[{\mdl[e]{D}}]{S}}:\begin{align*}
		\funeq{\mph[e]{h}}{\fun{\mdl{P}}{\mdl[e]{D}}}{\fun{\mdl{P}}{\mdl[{\mdl[e]{D}}]{S}}}.
	\end{align*}
\end{corollary}

There is nothing intrinsic to geometric properties (\eg, alignment, magnitude) in \mdl[{\mdl[e]{D}}]{S} that necessarily ties them together independently of their shared source domain relation; the set relation in the source domain \mdl[e]{D} determines the structure of \f{E'}, and any commonality among the vectors in \f{E'} is entirely a reflection of the original subset \f{E}. Thus, the extensional semantics drives the correspondences above; the vector space serves as a faithful representation.

We may further elaborate and extend to other structures such as the characteristic function\footnote{How truth vectors are represented in \mdl[{\mdl{S}}]{M} is discussed below; suffice it to say that we may represent them as basis vectors or as binary-valued functions.} to simply probe for set membership and indicator vectors to represent set membership in a more compact form.

\begin{corollary}
	Define the \term{characteristic function} \f{\chi_S:\mdl[e]{D}\rightarrow\bc{0,1}} by:\begin{align*}
		\fun{\chi_S}{e}= \begin{cases}1, &  e \in S \\ 0, &  e \notin S\end{cases}
	\end{align*}
	
	Similarly, define \f{\chi_{S'}:\mdl[{\mdl[e]{D}}]{S}\rightarrow\bc{0,1}} by:\begin{align*}
		\fun{\chi_{S'}}{\vv{v}}= \begin{cases}1, &  \vv{v} \in S' \\ 0, &  \vv{v} \notin S'\end{cases}
	\end{align*}
	
	The characteristic functions satisfy:\begin{align*}
		\funeq{\chi_S}{e}{\fun{\chi_{S'}}{\fun{\mph[e]{h}}{e}}}.
	\end{align*}

\end{corollary}

\begin{corollary}
	For finite \mdl[e]{D}, subsets \f{S} can be represented by \term{indicator vectors} in \mdl[{\mdl[e]{D}}]{S}, which are vectors that have a \f{1} in positions corresponding to elements in the subset and \f{0} elsewhere.
	
	Each entity corresponds to a basis vector \vv[i]{b} and a subset \f{S} corresponds to the sum:\begin{align*}
		\vv[S]{v}=\sum_{e \in S} \fun{\mph[e]{h}}{e}.
	\end{align*}

\end{corollary}

The mapping \mph[e]{h} preserves, as seen, set-membership, which, in turn, transfers properties, operations, and relations that involve subsets in \mdl[e]{D} into \mdl[{\mdl{S}}]{M}. Having shown this, we may access structures and properties defined over \mdl[e]{D} necessary for \mdl[{\mdl{S}}]{M}, such as measures, distributions, and sigma-algebras.

\subsection{Proof for \mph[t]{h}}

The characterization of \mph[t]{h} is consistent with the vectorization of truth values according to vector logic \cite{WestphalHardy2005,MIZRAJI1992179,mizr2008}. Let \f{\mdl[t]{D} = \bc{1,0}} be the set of truth values\footnote{In principle, we may expand beyond a binary truth-value system; the idea remains the same. We remain with a boolean truth value here for simplicity.} in the extensional model \mdl[ext]{M}. Let \mdl[{\mdl[t]{D}}]{S} be a \f{q}-dimensional (real) vector space \ff[q]{R}, where\footnote{Technically, the choice of \f{q} is a design decision; for simplicity, we consider \f{q=1,2}, though, in principle, higher dimensions are also acceptable.} \f{q\geq 1}. We define two distinct, normalized vectors in \mdl[{\mdl[t]{D}}]{S}: \bc{\vv[1]{b},\vv[0]{b}}, where \vv[1]{b} and \vv[0]{b} are orthonormal basis vectors representing true and false, respectively. These vectors satisfy normalization:\begin{align*}
	\abs{\vv[1]{b}} = 1, \quad \abs{\vv[0]{b}} = 1,
\end{align*}

where \abs{\cdot} denotes the Euclidean norm, and satisfy orthogonality:\begin{align*}
	\inn{\vv[1]{b}}{\vv[0]{b}} = 0,
\end{align*}

where \inn{\cdot}{\cdot} denotes the inner product in \ff[q]{R}.

Let \f{\mdl[t]{D} = \bc{1,0}}. We define the mapping \f{\mph[t]{h}:\mdl[t]{D}\rightarrow\img{\mph[t]{h}}\subseteq\mdl[{\mdl[t]{D}}]{S}} by:\begin{align*}
	\fun{\mph[t]{h}}{1}=\vv[1]{b}, \quad \fun{\mph[t]{h}}{0}=\vv[0]{b},
\end{align*}

which, for \f{q=2}, has the following (normalized and orthogonal) forms:\begin{align*}
	\fun{\mph[t]{h}}{1}=\vv[1]{b} =\left[\begin{array}{c}
		1 \\
		0
	\end{array}\right], \quad \fun{\mph[t]{h}}{0}=\vv[0]{b}=\left[\begin{array}{c}
		0 \\
		1
	\end{array}\right].
\end{align*}

\begin{lemma}\label{lem:bijt}
	The mapping \f{\mph[t]{h}:\mdl[t]{D}\rightarrow\img{\mph[t]{h}}\subseteq\mdl[{\mdl[t]{D}}]{S}} is an injective mapping between \mdl[t]{D} and \mdl[{\mdl[t]{D}}]{S}.
\end{lemma}

\begin{proof}
	
	Suppose \f{t_1,t_2 \in \mdl[t]{D}} and \f{t_1\neq t_2}. Since \f{\mdl[t]{D}=\bc{1,0}}, the possible distinct pairs are either \bp{1,0} or \bp{0,1}. It follows that \f{\fun{\mph[t]{h}}{t_1}\neq\fun{\mph[t]{h}}{t_2}} since \f{\fun{\mph[t]{h}}{0}\neq\fun{\mph[t]{h}}{1}} (our truth values are mutually exclusive) and \f{\vv[1]{b}\neq\vv[0]{b}} (orthonormal and distinct). Therefore, since \f{\fun{\mph[t]{h}}{t_1}\neq\fun{\mph[t]{h}}{t_2}} whenever \f{t_1 \neq t_2}, the mapping \mph[t]{h} is injective.

	\end{proof}

\begin{lemma}\label{lem:fun}
	For every \f{n}-ary logical function \f{f:\mdl[t]{D}^n\rightarrow\mdl[t]{D}} in the extensional model \mdl[ext]{M}, there exists a function \f{f':\img{\mph[t]{h}}^n \rightarrow \img{\mph[t]{h}}} in the vector space model \mdl[{\mdl{S}}]{M} such that for all \f{t_1,t_2,\dots,t_n\in\mdl[t]{D}}:\begin{align*}
		\funeq{\mph[t]{h}}{\fun{f}{t_1,t_2,\dots,t_n}}{\fun{f'}{\fun{\mph[t]{h}}{t_1},\fun{\mph[t]{h}}{t_2},\dots,\fun{\mph[t]{h}}{t_n}}},
	\end{align*}
	
	where \f{\mph[t]{h}:\mdl[t]{D}\rightarrow\img{\mph[t]{h}}\subseteq\mdl[{\mdl[t]{D}}]{S}}, and \mdl[{\mdl[t]{D}}]{S} is taken to be \ff[q]{R}.

\end{lemma}

\begin{proof}\label{pr:fun}
	Recall from Lemma~\ref{lem:bijt}, we have established that \f{\mdl[t]{D} = \bc{1,0}} representing truth values true and false. We define our injective mapping \f{\mph[t]{h}:\mdl[t]{D}\rightarrow\img{\mph[t]{h}}\subseteq\ff[q]{R}}, and let \f{\img{\mph[t]{h}}=\bc{\vv[1]{b},\vv[0]{b}}}, where:\begin{align*}
		\fun{\mph[t]{h}}{1} = \vv[1]{b}=\left[\begin{array}{l}
			1 \\
			0
		\end{array}\right], \quad \fun{\mph[t]{h}}{0} = \vv[0]{b}=\left[\begin{array}{l}
			0 \\
			1
		\end{array}\right],
	\end{align*}
	
	and since \mph[t]{h} is injective, it has a left-inverse \f{\mph[t]{h}^{-1}} well-defined on \img{\mph[t]{h}}.
	
	We define \f{f':\img{\mph[t]{h}}^n \rightarrow \img{\mph[t]{h}}} such that for all \f{\vv[1]{v},\vv[2]{v},\dots,\vv[n]{v}\in\img{\mph[t]{h}}} by:\begin{align*}
		\funeq{f'}{\vv[1]{v},\vv[2]{v},\dots,\vv[n]{v}}{\fun{\mph[t]{h}}{\fun{f}{\fun{\mph[t]{h}^{-1}}{\vv[1]{v}},\fun{\mph[t]{h}^{-1}}{\vv[2]{v}},\dots,\fun{\mph[t]{h}^{-1}}{\vv[n]{v}}}}},
	\end{align*}
	
	or, laboriously:\begin{align*}
		\funeq{f'}{\fun{\mph[t]{h}}{t_1},\fun{\mph[t]{h}}{t_2},\dots,\fun{\mph[t]{h}}{t_n}}{\fun{\mph[t]{h}}{\fun{f}{\fun{\mph[t]{h}^{-1}}{\fun{\mph[t]{h}}{t_1}},\fun{\mph[t]{h}^{-1}}{\fun{\mph[t]{h}}{t_2}},\dots,\fun{\mph[t]{h}^{-1}}{\fun{\mph[t]{h}}{t_n}}}}}.
	\end{align*}
	
	We show directly that, for all \f{t_1,t_2,\dots,t_n\in\mdl[t]{D}}:\begin{align*}
		\fun{\mph[t]{h}}{\fun{f}{t_1,t_2,\dots,t_n}} & = \fun{\mph[t]{h}}{\fun{f}{\fun{\mph[t]{h}^{-1}}{\fun{\mph[t]{h}}{t_1}},\fun{\mph[t]{h}^{-1}}{\fun{\mph[t]{h}}{t_2}},\dots,\fun{\mph[t]{h}^{-1}}{\fun{\mph[t]{h}}{t_n}}}}\\
		& = \fun{f'}{\fun{\mph[t]{h}}{t_1},\fun{\mph[t]{h}}{t_2},\dots,\fun{\mph[t]{h}}{t_n}} 
		,
	\end{align*}
	
	and we have that: \f{\funeq{\mph[t]{h}}{\fun{f}{t_1,t_2,\dots,t_n}}{\fun{f'}{\vv[1]{v},\vv[2]{v},\dots,\vv[n]{v}} } = \fun{f'}{\fun{\mph[t]{h}}{t_1},\fun{\mph[t]{h}}{t_2},\dots,\fun{\mph[t]{h}}{t_n}}}. Therefore, the function \f{f'} is well-defined on \f{\img{\mph[t]{h}}^n} and maps into \img{\mph[t]{h}}, preserving the operation \f{f} within the vector space model.

\end{proof}

While \f{f'} is defined via \mph[t]{h} and \f{f}, we can consider how \f{f'} operates within \mdl[{\mdl[t]{D}}]{S}. Given that \f{\mdl[{\mdl[t]{D}}]{S}=\bc{\vv[1]{b},\vv[0]{b}}}, and that \mph[t]{h} maps truth values to these vectors, \f{f'} essentially maps tuples of these basis vectors to one of the basis vectors.

\begin{remark}
	Since \img{\mph[t]{h}} is not a full vector space (here, it is a set of two vectors), we can consider what extending \f{f'} might look like. Since \f{f'} only operates on elements of \img{\mph[t]{h}}, our \f{f'} is fully defined for all necessary inputs. If we were to consider \mdl[{\mdl[t]{D}}]{S} as the full space \ff[2]{R}, then \f{f'} would need to be defined on vectors outside of \img{\mph[t]{h}}, which may not correspond to truth values, and so lose meaningful semantics.
\end{remark}

The design of the mapping \mph[t]{h} is such that it is compatible with representing monadic and dyadic logical operators using matrices in vector logic. To illustrate this and verify that these representations respect the framework established in Lemma~\ref{lem:fun}, we proceed with some definitions and examples of monadic, dyadic, triadic (ternary), and \f{n}-ary operators represented as matrices acting on the vector representations of truth values. See \cite{mizr2008,MIZRAJI1992179} for standard references, \cite{mizraji1996} for derivations of monadic and dyadic operators, and \cite{WestphalHardy2005} for elaboration and useful visualizations.

\begin{definition}\label{def:mon}
	A \term{monadic operator} \f{\mon:\mdl[t]{D}\rightarrow\mdl[t]{D}} in the extensional model \mdl[ext]{M} is a rule defined by a truth table with one argument, and is represented in the vector space model \mdl[{\mdl{S}}]{M} as \f{\mon'} by a \f{2\times 2} matrix \f{M} such that for all \f{t\in\mdl[t]{D}}:\begin{align*}
		\fun{\mon'}{\fun{\mph[t]{h}}{t}}=M\cdot \fun{\mph[t]{h}}{t} = \fun{\mph[t]{h}}{\fun{\mon}{t}}.
	\end{align*}
\end{definition}

Definition~\ref{def:mon} is essentially a specification of the representation of a function shown in Lemma~\ref{lem:fun}, where we established that for any logical function \f{f:\mdl[t]{D}\rightarrow\mdl[t]{D}}, there exists an operation \f{f':\mdl[{\mdl[t]{D}}]{S}\rightarrow\mdl[{\mdl[t]{D}}]{S}} satisfying: \funeq{\mph[t]{h}}{\fun{f}{t}}{\fun{f'}{\fun{\mph[t]{h}}{t}}}. An example of a monadic operator is negation, shown in Example~\ref{exe:neg}.

\begin{example}\label{exe:neg}
	In the extensional model \mdl[ext]{M}, the negation operator \f{\neg:\mdl[t]{D}\rightarrow\mdl[t]{D}} is defined by, for \f{t \mapsto s} (our \fun{\mph[t]{h}}{1}) and \f{f \mapsto n} (our \fun{\mph[t]{h}}{0}):\begin{align*}
		\funeq{\neg}{1}{0}, \quad \funeq{\neg}{0}{1},
	\end{align*}
	
	which corresponds to the table:\begin{table}[H]
		\centering
		\begin{tabular}{@{}l|l@{}}
			\f{\neg}	&  \\ \midrule
			1	& 0  \\
			0	& 1  \\ 
		\end{tabular}
	\end{table}
	
	The negation operator corresponds to the matrix:\begin{align*}
		N=n s^{\top}+s n^{\top}=\left[\begin{array}{l}
			0 \\
			1
		\end{array}\right]\left[\begin{array}{ll}
			1 & 0
		\end{array}\right]+\left[\begin{array}{l}
			1 \\
			0
		\end{array}\right]\left[\begin{array}{ll}
			0 & 1
		\end{array}\right]=\left[\begin{array}{ll}
			0 & 1 \\
			1 & 0
		\end{array}\right] .
	\end{align*}
	
	For \f{t=1}:\begin{align*}
		\fun{\mon'}{\fun{\mph[t]{h}}{1}} & = N \cdot \fun{\mph[t]{h}}{1} \\
		& =\left[\begin{array}{ll}
			0 & 1 \\
			1 & 0
		\end{array}\right]\left[\begin{array}{l}
			1 \\
			0
		\end{array}\right] \\
		& = \left[\begin{array}{l}
			0 \\
			1
		\end{array}\right] \\
		& = \fun{\mph[t]{h}}{0} \\
		& = \fun{\mph[t]{h}}{\fun{\neg}{1}}.
	\end{align*}
	
	Similarly, for \f{t=0}:\begin{align*}
		\fun{\mon'}{\fun{\mph[t]{h}}{0}} & = N \cdot \fun{\mph[t]{h}}{0} \\
		& =\left[\begin{array}{ll}
			0 & 1 \\
			1 & 0
		\end{array}\right]\left[\begin{array}{l}
			0 \\
			1
		\end{array}\right] \\
		& = \left[\begin{array}{l}
			1 \\
			0
		\end{array}\right] \\
		& = \fun{\mph[t]{h}}{1} \\
		& = \fun{\mph[t]{h}}{\fun{\neg}{0}}.
	\end{align*}
\end{example}

\begin{definition}\label{def:dya}
	A \term{dyadic operator} \f{\dya:\bp{\mdl[t]{D}\times\mdl[t]{D}}\rightarrow\mdl[t]{D}} in the extensional model \mdl[ext]{M} is a rule defined by a truth table with two arguments, and is represented in the vector space model \mdl[{\mdl{S}}]{M} as \f{\dya'} by a \f{4\times 2} matrix \f{M} such that for all \f{t_1,t_2\in\mdl[t]{D}}:\begin{align*}
		\fun{\dya'}{\fun{\mph[t]{h}}{t_1},\fun{\mph[t]{h}}{t_2}}=M\cdot\bp{\fun{\mph[t]{h}}{t_1}\otimes\fun{\mph[t]{h}}{t_2}}  = \fun{\mph[t]{h}}{\fun{\dya}{t_1,t_2}},
	\end{align*}
	
	where \f{\otimes} denotes the tensor (Kronecker) product.
\end{definition}

\begin{example}\label{exe:dya}
	The conjunction operator \f{\wedge:\bp{\mdl[t]{D}\times\mdl[t]{D}}\rightarrow\mdl[t]{D}} is defined by:\begin{align*}
		\begin{aligned}
			&	\funeq{\wedge}{1,1}{1} \\
			&	\funeq{\wedge}{1,0}{0} \\
			&	\funeq{\wedge}{0,1}{0} \\
			&	\funeq{\wedge}{0,0}{0},
		\end{aligned}
	\end{align*}
	
	which corresponds to the table:\begin{table}[H]
		\centering
		\begin{tabular}{@{}l|ll@{}}
		\f{\wedge}	& 1 & 0 \\ \midrule
		1	& 1 & 0 \\
		0	& 0 & 0 \\ 
		\end{tabular}
	\end{table}
	
	The conjunction operator corresponds to the matrix:\begin{align*}
		C=s\bp{s \otimes s}^{\top}+n\bp{s \otimes n}^{\top}+n\bp{n \otimes s}^{\top}+n\bp{n \otimes n}^{\top},
	\end{align*}
	
	which, after some calculation, gives the form:\begin{align*}
		C=\left[\begin{array}{llll}
			1 & 0 & 0 & 0 \\
			0 & 1 & 1 & 1
		\end{array}\right].
	\end{align*}
	
	For \f{t_1,t_2\in\mdl[t]{D}}, we compute \f{\fun{\mph[t]{h}}{t_1}\otimes\fun{\mph[t]{h}}{t_2}}, and apply \f{C}. We consider all possible combinations.
	
	\begin{enumerate}
		\item For \f{t_1=1, t_2=1}, we have that \funeq{\mph[t]{h}}{1}{s} and \funeq{\mph[t]{h}}{1}{s} whose tensor product is \f{s\otimes s = \left[\begin{array}{l}
				1 \\
				0 \\
				0 \\
				0
			\end{array}\right]}, and the application of the operator \f{C} follows as:\begin{align*}
			\fun{\dya'}{s,s}=C \cdot\bp{s\otimes s}=\left[\begin{array}{llll}
				1 & 0 & 0 & 0 \\
				0 & 1 & 1 & 1
			\end{array}\right]\left[\begin{array}{l}
				1 \\
				0 \\
				0 \\
				0
			\end{array}\right]=\left[\begin{array}{l}
				1 \\
				0
			\end{array}\right]=s=\fun{\mph[t]{h}}{1} = \fun{\mph[t]{h}}{\fun{\wedge}{1,1}} .
		\end{align*}
		
		\item For \f{t_1=1, t_2=0}, we have that \funeq{\mph[t]{h}}{1}{s} and \funeq{\mph[t]{h}}{0}{n} whose tensor product is \f{s\otimes n = \left[\begin{array}{l}
				0 \\
				1 \\
				0 \\
				0
			\end{array}\right]}, and the application of the operator \f{C} follows as:\begin{align*}
			\fun{\dya'}{s,n}=C \cdot\bp{s\otimes n}=\left[\begin{array}{llll}
				1 & 0 & 0 & 0 \\
				0 & 1 & 1 & 1
			\end{array}\right]\left[\begin{array}{l}
				0 \\
				1 \\
				0 \\
				0
			\end{array}\right]=\left[\begin{array}{l}
				0 \\
				1
			\end{array}\right]=n=\fun{\mph[t]{h}}{0} = \fun{\mph[t]{h}}{\fun{\wedge}{1,0}} .
		\end{align*}
		
		\item For \f{t_1=0, t_2=1}, we have that \funeq{\mph[t]{h}}{0}{n} and \funeq{\mph[t]{h}}{1}{s} whose tensor product is \f{n\otimes s = \left[\begin{array}{l}
				0 \\
				0 \\
				1 \\
				0
			\end{array}\right]}, and the application of the operator \f{C} follows as:\begin{align*}
			\fun{\dya'}{n,s}=C \cdot\bp{n\otimes s}=\left[\begin{array}{llll}
				1 & 0 & 0 & 0 \\
				0 & 1 & 1 & 1
			\end{array}\right]\left[\begin{array}{l}
				0 \\
				0 \\
				1 \\
				0
			\end{array}\right]=\left[\begin{array}{l}
				0 \\
				1
			\end{array}\right]=n=\fun{\mph[t]{h}}{0} = \fun{\mph[t]{h}}{\fun{\wedge}{0,1}} .
		\end{align*}
		
		\item For \f{t_1=0, t_2=0}, we have that \funeq{\mph[t]{h}}{0}{n} and \funeq{\mph[t]{h}}{0}{n} whose tensor product is \f{n\otimes n = \left[\begin{array}{l}
				0 \\
				0 \\
				0 \\
				1
			\end{array}\right]}, and the application of the operator \f{C} follows as:\begin{align*}
			\fun{\dya'}{n,n}=C \cdot\bp{n\otimes n}=\left[\begin{array}{llll}
				1 & 0 & 0 & 0 \\
				0 & 1 & 1 & 1
			\end{array}\right]\left[\begin{array}{l}
				0 \\
				0 \\
				0 \\
				1
			\end{array}\right]=\left[\begin{array}{l}
				0 \\
				1
			\end{array}\right]=n=\fun{\mph[t]{h}}{0} = \fun{\mph[t]{h}}{\fun{\wedge}{0,0}} .
		\end{align*}
	\end{enumerate}
	
\end{example}

\begin{example}\label{exe:dya2}
	Similarly, other logical operators can be represented in \mdl[{\mdl{S}}]{M} via matrices. The matrix \f{D} representing disjunction can be constructed such that:\begin{align*}
		D=s\bp{s \otimes s}^{\top}+s\bp{s \otimes n}^{\top}+s\bp{n \otimes s}^{\top}+n\bp{n \otimes n}^{\top} .
	\end{align*}
	
	The matrix \f{L} representing implication can be defined using the matrices for disjunction, negation, and the identity:\begin{align*}
		L=D \cdot\bp{N\otimes I}.
	\end{align*}
	
	Matrices for equivalence \f{E} and exclusive OR \f{X} can be defined similarly:\begin{align*}
		\begin{aligned}
			&	E=s\bp{s \otimes s}^{\top}+n\bp{s \otimes n}^{\top}+n\bp{n \otimes s}^{\top}+s\bp{n \otimes n}^{\top}, \\
			&	X=N E.
		\end{aligned}
	\end{align*}

\end{example}

\begin{definition}\label{def:tri}
	A \term{triadic} or \term{ternary operator} \f{\tri:\bp{\mdl[t]{D}\times\mdl[t]{D}\times\mdl[t]{D}}\rightarrow\mdl[t]{D}} in the extensional model \mdl[ext]{M} is a rule defined by a truth table with three arguments, and is represented in the vector space model \mdl[{\mdl{S}}]{M} as \f{\tri'} by an \f{8\times 2} matrix \f{M} such that for all \f{t_1,t_2,t_3\in\mdl[t]{D}}:\begin{align*}
		\fun{\tri'}{\fun{\mph[t]{h}}{t_1},\fun{\mph[t]{h}}{t_2},\fun{\mph[t]{h}}{t_3}}=M\cdot\bp{\fun{\mph[t]{h}}{t_1}\otimes\fun{\mph[t]{h}}{t_2}\otimes\fun{\mph[t]{h}}{t_3}}  = \fun{\mph[t]{h}}{\fun{\tri}{t_1,t_2,t_3}},
	\end{align*}
	
	where \f{\otimes} denotes the tensor (Kronecker) product.
\end{definition}

\begin{example}\label{exe:tri}
	The ternary conditional operator \cond{} (if-then-else) takes three truth values \f{t_1,t_2,t_3 \in \mdl[t]{D}} and returns a truth value based on the condition \f{t_1}:\begin{align*}
		\fun{\cond}{t_1, t_2, t_3}= \begin{cases}t_2 & \text { if } t_1=1 \\ t_3 & \text { if } t_1=0\end{cases}.
	\end{align*}
	
	For such triadic operators, we represent the combination of three truth values:\begin{align*}
		\fun{\mph[t]{h}}{t_1}\otimes\fun{\mph[t]{h}}{t_2}\otimes\fun{\mph[t]{h}}{t_3}\in\ff[2\times2\times2]{R},
	\end{align*}
	
	in which we flatten the tensor product into a vector in \ff[8]{R} as \vv[{t_1,t_2,t_3}]{v}. The conditional ternary operator, then corresponds to:\begin{align*}
		\fun{\tri'}{\fun{\mph[t]{h}}{t_1},\fun{\mph[t]{h}}{t_2},\fun{\mph[t]{h}}{t_3}} = M \cdot \vv[{t_1,t_2,t_3}]{v}=\fun{\mph[t]{h}}{\fun{\cond}{t_1,t_2,t_3}}.
	\end{align*}
	
	We need to define the matrix \f{M} such that it correctly maps each possible combination of \f{t_1,t_2,t_3} to the appropriate output, and since each \f{t_i} can be 0 or 1, there are \f{2^3 = 8} such possible combinations; we index the combinations as follows:\begin{table}[H]
		\centering
		\caption{Tabulated combinations for \f{M}}
		\label{tab:tabm}
		\begin{tabular}{@{}llllll@{}}
			\toprule
			\f{j}-th index	& \f{t_1} & \f{t_2} & \f{t_3} & \fun{\cond}{t_1,t_2,t_3} & \fun{\mph[t]{h}}{\fun{\cond}{t_1,t_2,t_3}} \\ \midrule
			1	& 1 & 1 & 1 & \f{t_2=1} & \f{\left[\begin{array}{l}
					1 \\
					0
				\end{array}\right]} \\
			2	& 1 & 1 & 0 & \f{t_2=1} & \f{\left[\begin{array}{l}
					1 \\
					0
				\end{array}\right]} \\
			3	& 1 & 0 & 1 & \f{t_2=0} & \f{\left[\begin{array}{l}
					0 \\
					1
				\end{array}\right]} \\
			4	& 1 & 0 & 0 & \f{t_2=0} & \f{\left[\begin{array}{l}
					0 \\
					1
				\end{array}\right]} \\
			5	& 0 & 1 & 1 & \f{t_3=1} & \f{\left[\begin{array}{l}
					1 \\
					0
				\end{array}\right]} \\
			6	& 0 & 1 & 0 & \f{t_3=0} & \f{\left[\begin{array}{l}
					0 \\
					1
				\end{array}\right]} \\
			7	& 0 & 0 & 1 & \f{t_3=1} & \f{\left[\begin{array}{l}
					1 \\
					0
				\end{array}\right]} \\
			8	& 0 & 0 & 0 & \f{t_3=0} & \f{\left[\begin{array}{l}
					0 \\
					1
				\end{array}\right]} \\ \bottomrule
		\end{tabular}
	\end{table}

	For each combination, we define a column in \f{M^{\top}} corresponding to \vv[{t_1,t_2,t_3}]{v}, mapping to \fun{\mph[t]{h}}{\fun{\cond}{t_1,t_2,t_3}}. Let \vv[i]{e} be the standard basis vectors in \ff[8]{R}, where \f{i} ranges from \f{1} to \f{8}. We then define \f{M} such that for \f{t=1} (cases 1-4), \funeq{\cond}{t_1,t_2,t_3}{t_2}, and for \f{t=0} (cases 5-8), \funeq{\cond}{t_1,t_2,t_3}{t_3}.
	
	Therefore, for each index \f{j}, the \f{j}-th column of \f{M} corresponds to \fun{\mph[t]{h}}{\fun{\cond}{t_1,t_2,t_3}}, and we can write the matrix \f{M} explicitly as:\begin{align*}
		M=\left[\begin{array}{llllllll}
			M_{1,1} & M_{1,2} & M_{1,3} & M_{1,4} & M_{1,5} & M_{1,6} & M_{1,7} & M_{1,8} \\
			M_{2,1} & M_{2,2} & M_{2,3} & M_{2,4} & M_{2,5} & M_{2,6} & M_{2,7} & M_{2,8}
		\end{array}\right],
	\end{align*}
	
	each \f{M_{ij}} corresponds to the mapping from \f{e_j} to \fun{\mph[t]{h}}{\fun{\cond}{t_1,t_2,t_3}}. Then, for each \f{j}, we determine \f{t_1,t_2,t_3} corresponding to the index \f{j}, and compute \fun{\cond}{t_1,t_2,t_3}; so, we set \f{M_{:,j}=\fun{\mph[t]{h}}{\fun{\cond}{t_1,t_2,t_3}}}.
	
	For the first case that \f{t_1=1,t_2=1,t_3=1}, we have that \funeq{\cond}{1,1,1}{1}, and \f{M_{:, 1}=\fun{\mph[t]{h}}{1}=\left[\begin{array}{l}
			1 \\
			0
		\end{array}\right]}. Likewise, for the fifth case that \f{t_1=0,t_2=1,t_3=1}, we have that \funeq{\cond}{0,1,1}{1}, and \f{M_{:, 5}=\fun{\mph[t]{h}}{1}=\left[\begin{array}{l}
		1 \\
		0
		\end{array}\right]}. In this way, \f{M_{1j}} is the first element of \fun{\mph[t]{h}}{\fun{\cond}{t_1,t_2,t_3}} for the \f{j}-th combination, and \f{M_{2j}} is the second element. By filling out all entries, we obtain the complete matrix \f{M}:\begin{align*}
		M=\left[\begin{array}{llllllll}
			M_{1,1} & M_{1,2} & M_{1,3} & M_{1,4} & M_{1,5} & M_{1,6} & M_{1,7} & M_{1,8} \\
			M_{2,1} & M_{2,2} & M_{2,3} & M_{2,4} & M_{2,5} & M_{2,6} & M_{2,7} & M_{2,8}
		\end{array}\right]=\left[\begin{array}{llllllll}
			1 & 1 & 0 & 0 & 1 & 0 & 1 & 0 \\
			0 & 0 & 1 & 1 & 0 & 1 & 0 & 1
		\end{array}\right]
	\end{align*}
	
	We apply this matrix in the following way. For each combination, we compute the tensor product:\begin{align*}
		\vv[t_1,t_2,t_3]{v}=\fun{\mph[t]{h}}{t_1}\otimes\fun{\mph[t]{h}}{t_2}\otimes\fun{\mph[t]{h}}{t_3} \in \ff[8]{R},
	\end{align*}
	
	and compute:\begin{align*}
		\funeq{\cond'}{\fun{\mph[t]{h}}{t_1},\fun{\mph[t]{h}}{t_2},\fun{\mph[t]{h}}{t_3}}{M\cdot \vv[t_1,t_2,t_3]{v}}.
	\end{align*}
	
	Computing for verification, we have:
	
	\begin{itemize}
		\item for the first case, where \f{t_1=1,t_2=1,t_3=1}, we have that \funeq{\mph[t]{h}}{1}{\left[\begin{array}{l}
				1 \\
				0
			\end{array}\right]}, \funeq{\mph[t]{h}}{1}{\left[\begin{array}{l}
			1 \\
			0
			\end{array}\right]}, and \funeq{\mph[t]{h}}{1}{\left[\begin{array}{l}
			1 \\
			0
			\end{array}\right]}, from which we compute \vv[1,1,1]{v}:\begin{align*}
			\vv[1,1,1]{v}=\fun{\mph[t]{h}}{1} \otimes \fun{\mph[t]{h}}{1} \otimes \fun{\mph[t]{h}}{1}=\left[\begin{array}{l}
				1 \\
				0 \\
				0 \\
				0 \\
				0 \\
				0 \\
				0 \\
				0
			\end{array}\right].
		\end{align*}
		
		Applying the matrix \f{M}, we find that:\begin{align*}
			M \cdot \vv[1,1,1]{v}=\left[\begin{array}{llllllll}
				1 & 1 & 0 & 0 & 1 & 0 & 1 & 0 \\
				0 & 0 & 1 & 1 & 0 & 1 & 0 & 1
			\end{array}\right]\left[\begin{array}{l}
				1 \\
				0 \\
				0 \\
				0 \\
				0 \\
				0 \\
				0 \\
				0
			\end{array}\right]=\left[\begin{array}{l}
				1 \\
				0
			\end{array}\right],
		\end{align*}
		
		which we compare \fun{\mph[t]{h}}{\fun{\cond}{1,1,1}}:\begin{align*}
			\fun{\mph[t]{h}}{\fun{\cond}{1,1,1}}=\fun{\mph[t]{h}}{1}=\left[\begin{array}{l}
				1 \\
				0
			\end{array}\right].
		\end{align*}
		
		The output matches, and so confirms that \f{M} correctly represents the operator for the first case.
		
		\item for the third case, where \f{t_1=1,t_2=0,t_3=1}, we have that \funeq{\mph[t]{h}}{1}{\left[\begin{array}{l}
				1 \\
				0
			\end{array}\right]}, \funeq{\mph[t]{h}}{0}{\left[\begin{array}{l}
				0 \\
				1
			\end{array}\right]}, and \funeq{\mph[t]{h}}{1}{\left[\begin{array}{l}
				1 \\
				0
			\end{array}\right]}, from which we compute \vv[1,0,1]{v}:\begin{align*}
			\vv[1,0,1]{v}=\fun{\mph[t]{h}}{1} \otimes \fun{\mph[t]{h}}{0} \otimes \fun{\mph[t]{h}}{1}=\left[\begin{array}{l}
				0 \\
				0 \\
				1 \\
				0 \\
				0 \\
				0 \\
				0 \\
				0
			\end{array}\right].
		\end{align*}
		
		Applying the matrix \f{M}, we find that:\begin{align*}
			M \cdot \vv[1,0,1]{v}=\left[\begin{array}{llllllll}
				1 & 1 & 0 & 0 & 1 & 0 & 1 & 0 \\
				0 & 0 & 1 & 1 & 0 & 1 & 0 & 1
			\end{array}\right]\left[\begin{array}{l}
				0 \\
				0 \\
				1 \\
				0 \\
				0 \\
				0 \\
				0 \\
				0
			\end{array}\right]=\left[\begin{array}{l}
				0 \\
				1
			\end{array}\right],
		\end{align*}
		
		which we compare \fun{\mph[t]{h}}{\fun{\cond}{1,0,1}}:\begin{align*}
			\fun{\mph[t]{h}}{\fun{\cond}{1,0,1}}=\fun{\mph[t]{h}}{0}=\left[\begin{array}{l}
				0 \\
				1
			\end{array}\right].
		\end{align*}
		
		The output matches, and so confirms that \f{M} correctly represents the operator for the third case.
		
		\item for the fifth case, where \f{t_1=0,t_2=1,t_3=1}, we have that \funeq{\mph[t]{h}}{0}{\left[\begin{array}{l}
				0 \\
				1
			\end{array}\right]}, \funeq{\mph[t]{h}}{1}{\left[\begin{array}{l}
				1 \\
				0
			\end{array}\right]}, and \funeq{\mph[t]{h}}{1}{\left[\begin{array}{l}
				1 \\
				0
			\end{array}\right]}, from which we compute \vv[0,1,1]{v}:\begin{align*}
			\vv[0,1,1]{v}=\fun{\mph[t]{h}}{0} \otimes \fun{\mph[t]{h}}{1} \otimes \fun{\mph[t]{h}}{1}=\left[\begin{array}{l}
				0 \\
				0 \\
				0 \\
				0 \\
				1 \\
				0 \\
				0 \\
				0
			\end{array}\right].
		\end{align*}
		
		Applying the matrix \f{M}, we find that:\begin{align*}
			M \cdot \vv[0,1,1]{v}=\left[\begin{array}{llllllll}
				1 & 1 & 0 & 0 & 1 & 0 & 1 & 0 \\
				0 & 0 & 1 & 1 & 0 & 1 & 0 & 1
			\end{array}\right]\left[\begin{array}{l}
				0 \\
				0 \\
				0 \\
				0 \\
				1 \\
				0 \\
				0 \\
				0
			\end{array}\right]=\left[\begin{array}{l}
				1 \\
				0
			\end{array}\right],
		\end{align*}
		
		which we compare \fun{\mph[t]{h}}{\fun{\cond}{1,0,1}}:\begin{align*}
			\fun{\mph[t]{h}}{\fun{\cond}{0,1,1}}=\fun{\mph[t]{h}}{1}=\left[\begin{array}{l}
				1 \\
				0
			\end{array}\right].
		\end{align*}
		
		The output matches, and so confirms that \f{M} correctly represents the operator for the fifth case.

	\end{itemize}
	
	Similarly, for the remaining cases, we can verify and confirm that the matrix \f{M} consistently produces the correct output for the ternary conditional operator.

\end{example}

\begin{definition}\label{def:nop}
	An \term{\f{n}-ary operator} \f{\nry:\mdl[t]{D}^n\rightarrow\mdl[t]{D}} in the extensional model \mdl[ext]{M} is a rule defined by a truth table with \f{n} arguments, and is represented in the vector space model \mdl[{\mdl{S}}]{M} as \f{\nry'} by an \f{2^n \times 2} matrix \f{M} such that for all \f{t_1,t_2,\dots,t_n\in\mdl[t]{D}}:\begin{align*}
		\fun{\nry'}{\fun{\mph[t]{h}}{t_1},\fun{\mph[t]{h}}{t_2},\dots,\fun{\mph[t]{h}}{t_n}}=M\cdot\bp{\bigotimes_{i=1}^n \fun{\mph[t]{h}}{t_i}}  = \fun{\mph[t]{h}}{\fun{\nry}{t_1,t_2,\dots,t_n}},
	\end{align*}
	
	where \f{\otimes} denotes the tensor (Kronecker) product.
\end{definition}

\begin{remark}\label{rmk:recipe}
	The recipe for a general \f{n}-ary operator is as follows:
	
	\begin{enumerate}
		\item list all \f{2^n} possible combinations of truth values \bp{t_1,t_2,\dots,t_n}
		\item for each combination:
		\begin{enumerate}
			\item compute \fun{\nry}{t_1,t_2,\dots,t_n}
			\item compute the tensor product \f{\vv[t_1,t_2,\dots,t_n]{v}=\bigotimes_{i=1}^n \fun{\mph[t]{h}}{t_i}}
			\item define the corresponding column in \f{M} such that:\begin{align*}
				M_{:, j}=\fun{\mph[t]{h}}{\fun{\nry}{t_1,t_2,\dots,t_n}},
			\end{align*}
			where \f{j} is the index of the combination in the enumeration
		\end{enumerate}
		
		\item the operator \f{\nry'} applies the matrix \f{M} to the tensor product:\begin{align*}
			\funeq{\nry'}{\fun{\mph[t]{h}}{t_1},\fun{\mph[t]{h}}{t_2},\dots,\fun{\mph[t]{h}}{t_n}}{M\cdot \vv[t_1,t_2,\dots,t_n]{v}}.
		\end{align*}
	\end{enumerate}
\end{remark}

\begin{example}\label{exe:nop}

	Consider the majority operator \maj{n} for arbitrary \f{n} in the vector space model. The majority operator returns \f{1} if the majority of its \f{n} inputs are \f{1}, and \f{0} otherwise. Such a matrix is \f{M\in\ff[2\times 2^n]{R}} that represents \maj{n} follows as:\begin{align*}
		\fun{\maj{n}'}{\fun{\mph[t]{h}}{t_1},\fun{\mph[t]{h}}{t_2},\dots,\fun{\mph[t]{h}}{t_n}}=M\cdot\bp{\bigotimes_{i=1}^n \fun{\mph[t]{h}}{t_i}}  = \fun{\mph[t]{h}}{\fun{\maj{n}}{t_1,t_2,\dots,t_n}},
	\end{align*}
	
	The majority operator \f{\maj{n}:\mdl[t]{D}^{n}\rightarrow\mdl[t]{D}} in the extensional model is defined by:\begin{align*}
		\fun{\maj{n}}{t_1,t_2,\dots,t_n}= \begin{cases}1 & \text { if } \sum_{i=1}^n t_i>\frac{n}{2} \\ 0 & \text { otherwise, }\end{cases}
	\end{align*}
	
	where \f{t_i\in\mdl[t]{D}} as usual, and the sum \f{\sum_{i=1}^n t_i} counts the number of ones in the input.
	
	The tensor product of \f{n} truth value vectors results in a \f{2^n}-dimensional vector of the form: \f{\vv[t_1,t_2,\dots,t_n]{v}=\bigotimes_{i=1}^n \fun{\mph[t]{h}}{t_i} \in \ff[2^n]{R}}. The matrix \f{M\in\ff[2\times 2^n]{R}} itself is then constructed such that each column corresponds to one of the \f{2^n} possible combinations of input truth values, of which are the vector representations \fun{\mph[t]{h}}{\fun{\maj{n}}{t_1,t_2,\dots,t_n}}.
	
	We construct a rather compact representation for the matrix \f{M} following the recipe above. As stated, there are \f{2^n} possible combinations of \bp{t_1,t_2,\dots,t_n}, where each \f{t_i\in\bc{0,1}}; we index each such combination by an integer \f{j} ranging from \f{1} to \f{2^n}. For each combination, we compute \fun{\maj{n}}{t_1,t_2,\dots,t_n} by calculating the sum \f{\sum_{i=1}^n t_i}, and determine the output of \maj{n} by checking whether the sum is greater than \f{\frac{n}{2}}. For each combination \f{j}, we set the \f{j}-th column of \f{M} as:\begin{align*}
		M_{:, j}=\fun{\mph[t]{h}}{\fun{\maj{n}}{t_1, t_2, \dots, t_n}},
	\end{align*}
	such that \f{M_{1, j}=1} and \f{M_{2, j}=0} if the majority function outputs 1, and \f{M_{1, j}=0} and \f{M_{2, j}=1} if it outputs 0.
	
	\f{M} can be enormous. Due to its exponential size, we can express it compactly using indicator functions (which is essentially the trick with the basis vectors we used in Example~\ref{exe:tri}). Let \vv[j]{e} be the \f{j}-th standard basis vector in \ff[2^n]{R}. For each \f{j}, corresponding to input \bp{t_1^{(j)}, t_2^{(j)}, \dots, t_n^{(j)}}, we define:\begin{align*}
		m_j = \fun{\maj{n}}{t_1^{(j)}, t_2^{(j)}, \dots, t_n^{(j)}} \in\bc{0,1}.
	\end{align*}
	
	In this way, the matrix \f{M} has the following representation:\begin{align*}
		M=\sum_{j=1}^{2^n} \fun{\mph[t]{h}}{m_j}\cdot \vv[j]{e}^{\top}.
	\end{align*}
	
\end{example}

From Definitions~\ref{def:mon}, \ref{def:dya}, \ref{def:tri}, \ref{def:nop}, as well as Examples~\ref{exe:neg}, \ref{exe:dya}, \ref{exe:dya2}, \ref{exe:tri}, \ref{exe:nop}, we see how logical operators in the extensional model relate to operations in the vector space model. Caution, however, should be exercised: in general, \f{n}-ary operators can explode exponentially from large \f{n}. For small \f{n}, constructing the matrix form of the \f{n}-ary majority operator \f{M} explicitly, for example, is feasible:

\begin{itemize}
	\item for \maj{n=3}, we have for any input truth values \f{t_1,t_2,t_3\in\mdl[t]{D}}:\begin{align*}
		\fun{\maj{3}'}{\fun{\mph[t]{h}}{t_1},\fun{\mph[t]{h}}{t_2},\fun{\mph[t]{h}}{t_3}}=M \cdot\bp{\bigotimes_{i=1}^3 \fun{\mph[t]{h}}{t_i}}=\fun{\mph[t]{h}}{\fun{\maj{3}}{t_1,t_2,t_3}} .
	\end{align*}

	and the matrix form is given by:\begin{align*}
		M=\left[\begin{array}{llllllll}
			1 & 1 & 1 & 0 & 1 & 0 & 0 & 0 \\
			0 & 0 & 0 & 1 & 0 & 1 & 1 & 1
		\end{array}\right]
	\end{align*}
	\item for \maj{n=4}, we have for any input truth values \f{t_1,t_2,t_3,t_4\in\mdl[t]{D}}:\begin{align*}
		\fun{\maj{4}'}{\fun{\mph[t]{h}}{t_1},\fun{\mph[t]{h}}{t_2},\fun{\mph[t]{h}}{t_3},\fun{\mph[t]{h}}{t_4}}=M \cdot\bp{\bigotimes_{i=1}^4 \fun{\mph[t]{h}}{t_i}}=\fun{\mph[t]{h}}{\fun{\maj{4}}{t_1,t_2,t_3,t_4}}.
	\end{align*}
	
	and the matrix form is given by:\begin{align*}
		M=\left[\begin{array}{llllllllllllllll}
			1 & 1 & 1 & 0 & 1 & 0 & 0 & 0 & 1 & 0 & 0 & 0 & 0 & 0 & 0 & 0 \\
			0 & 0 & 0 & 1 & 0 & 1 & 1 & 1 & 0 & 1 & 1 & 1 & 1 & 1 & 1 & 1
		\end{array}\right]
	\end{align*}
\end{itemize}

At some point, then, in the implementation, a decision needs to be made to mitigate the exponential explosion, even indulging in tricks otherwise outside of the scope of relating formal and vector frameworks; what that point is and what alternative form for the operator is are both design decisions. For example, the majority function may be expressed alternatively as a Boolean polynomial for Boolean variables \f{t_i}, which expands the function into a sum over combinations where the majority criterion is met \cite{odonnell2021analysisbooleanfunctions}; we still run into the issue of an exponential number of terms, however, though a polynomial expression may be recovered:\begin{align*}
	\fun{\maj{n}}{t_1, t_2, \dots, t_n}=\sum_{k=\ceil{\frac{n}{2}+1}}^n \sum_{\substack{I \subseteq\bc{1, \dots, n} \\\abs{I}=k}}\left(\prod_{i \in I} t_i \prod_{j \notin I}\left(1-t_j\right)\right),
\end{align*}

which we interpret in the following way\footnote{Specifically, the outer sum over \f{k} from \ceil{\frac{n}{2}+1} to \f{n} represents all cases where the number of ones exceeds half of \f{n}; the inner sum is over all subsets\f{I} of size \f{k} from the set \bc{1,2,\dots,n}; the product over \f{i \in I} checks for all variables in subset \f{I} to be 1; the product over \f{j \notin I} checks for all variables not in subset \f{I} to be 0. }: each term in the sum corresponds to a specific combination where exactly \f{k} variables are 1 and the rest are 0, and by summing over all such combinations where \f{k > \frac{n}{2}}, we capture all input configurations where the majority of variables are 1. 

For example, for \f{n=3}, the majority function outputs 1 when at least 2 of the 3 inputs are 1. \f{\ceil{\frac{3}{2}+1}=2.5}, so \f{k} ranges from 2 to 3. For \f{k=2}, we have subsets of size 2 given by: \bc{1,2}, \bc{1,3}, \bc{2,3}, and terms: \f{t_1 t_2\bp{1-t_3}}, \f{t_1 t_3\bp{1-t_2}}, \f{t_2 t_3\bp{1-t_1}}. Then for \f{k=3}, the only subset of size 3 is simply \bc{1,2,3}, and one term \f{t_1 t_2 t_3}. The final polynomial representation is, then:\begin{align*}
	\fun{\maj{3}}{t_1, t_2, t_3}=t_1 t_2\bp{1-t_3}+t_1 t_3\bp{1-t_2}+t_2 t_3\bp{1-t_1}+t_1 t_2 t_3 .
\end{align*}


Much of this seems to be reinventing the wheel as a square; the point of this discussion, however, is to illustrate how these operators may be represented in both \mdl[ext]{M} and \mdl[{\mdl{S}}]{M}, as well as to describe the process of how these operators may be related back and forth between these models. Resolving exponential explosion is an evergreen problem, not least of all in data representation and analysis and language processing, from which we are not exempt.

\subsection{Proof for \mph[f]{h}}

Let \f{f:\mdl[A]{D}\rightarrow \mdl[B]{D}} be an arbitrary semantic function in the extensional model \mdl[ext]{M}, where \mdl[A]{D} and \mdl[B]{D} are domains in \mdl[ext]{M}. Let \mdl[{\mdl[A]{D}}]{S} and \mdl[{\mdl[B]{D}}]{S} be corresponding spaces in \mdl[{\mdl{S}}]{M}. 

\begin{lemma}\label{lem:bijA}
	For any semantic domain \mdl[A]{D} in the extensional model \mdl[ext]{M}, the mapping \f{\mph[A]{h}:\mdl[A]{D}\rightarrow\img{\mph[A]{h}}\subseteq\mdl[{\mdl[A]{D}}]{S}} is an injective mapping between \mdl[A]{D} and \mdl[{\mdl[A]{D}}]{S} in the vector space model \mdl[{\mdl{S}}]{M}.
\end{lemma}

\begin{proof}\label{pr:bijA}
	Let \mdl[A]{D} be a semantic domain in the model \mdl[ext]{M}, and \mdl[{\mdl[A]{D}}]{S} a subset of the space \ff{R} of sufficient dimension; \mph[A]{h} is such that all \f{a\in\mdl[A]{D}} is assigned a unique vector \f{\fun{\mph[A]{h}}{a}\in\img{\mph[A]{h}}\subseteq\mdl[{\mdl[A]{D}}]{S}}. Suppose now that \f{\mdl[A]{D} = \bc{a_1,a_2,\dots,a_n}} with \f{n} elements, and we let \f{\mdl[{\mdl[A]{D}}]{S}=\ff[n]{R}}. We define \funeq{\mph[A]{h}}{a_i}{\vv[i]{v}}, where \vv[i]{v} is the \f{i}-th standard basis vector in \ff[n]{R}. Let \mdl[{\mdl[A]{D}}]{S} be a vector space of dimension equal to the cardinality of \mdl[A]{D}. We choose a basis representation of \bcdef{\vv[a]{v}}{a\in\mdl[A]{D} }, from which we define \funeq{\mph[A]{h}}{a}{\vv[a]{v}}.
	
	Assume now that \f{a_1,a_2 \in \mdl[A]{D}} such that \funeq{\mph[A]{h}}{a_1}{\fun{\mph[A]{h}}{a_2}}. By definition of \mph[A]{h}, we have that \funeq{\mph[A]{h}}{a_1}{\vv[a_1]{v}} and \funeq{\mph[A]{h}}{a_2}{\vv[a_2]{v}}, from which it follows that \f{\vv[a_1]{v} = \vv[a_2]{v}}. Since basis vectors are distinct, it follows that \f{a_1 = a_2}. Therefore, \mph[A]{h} is injective because it maps distinct elements of \mdl[A]{D} to distinct vectors in \mdl[{\mdl[A]{D}}]{S}. Furthermore, we have that the image \f{\img{\mph[A]{h}} = \bcdef{\fun{\mph[A]{h}}{a}}{a\in\mdl[A]{D}}\subseteq\mdl[{\mdl[A]{D}}]{S}}. For any vector \f{\vv{v}\in\img{\mph[A]{h}}}, there exists an \f{a\in\mdl[A]{D}} such that \f{\vv{v}=\fun{\mph[A]{h}}{a}}. Therefore, \mph[A]{h} is surjective onto its image \img{\mph[A]{h}}.
	
	\mph[A]{h}, then, is an injective mapping from \mdl[A]{D} into the space \mdl[{\mdl[A]{D}}]{S}, and the image \img{\mph[A]{h}} is a subset of \mdl[{\mdl[A]{D}}]{S}.

\end{proof}

\begin{lemma}\label{lem:func}
	For each semantic function \f{f: \mdl[A]{D}\rightarrow\mdl[B]{D}} in the extensional model \mdl[ext]{M}, there exists a function \f{\mph[f]{h}:\img{\mph[A]{h}}\rightarrow\img{\mph[B]{h}}} in the vector space model \mdl[{\mdl{S}}]{M} such that for all \f{a\in\mdl[A]{D}}:\begin{align*}
		\funeq{\mph[f]{h}}{\fun{\mph[A]{h}}{a}}{\fun{\mph[B]{h}}{\fun{f}{a}}},
	\end{align*}
	where \f{\mph[A]{h}:\mdl[A]{D}\rightarrow\img{\mph[A]{h}}\subseteq\mdl[{\mdl[A]{D}}]{S}} and \f{\mph[B]{h}:\mdl[B]{D}\rightarrow\img{\mph[B]{h}}\subseteq\mdl[{\mdl[B]{D}}]{S}} are injective mappings. 
\end{lemma}

The diagram commutes:

\begin{figure}[H]
	\[\begin{tikzcd}
		\mdl[A]{D} && \mdl[B]{D} \\
		\\
		\img{\mph[A]{h}} && \img{\mph[B]{h}}
		\arrow["{\mph{f}}", from=1-1, to=1-3]
		\arrow["{\mph[A]{h}}"', from=1-1, to=3-1]
		\arrow["{\mph[B]{h}}", from=1-3, to=3-3]
		\arrow["{\mph[f]{h}}"', from=3-1, to=3-3]
	\end{tikzcd}\]
\end{figure}

\begin{proof}\label{pr:func}
	Both \mph[A]{h} and \mph[B]{h} are injective mappings from their respective domains into vector spaces \mdl[{\mdl[A]{D}}]{S} and \mdl[{\mdl[B] {D}}]{S}, respectively. Since \mph[A]{h} is injective, it has a left-inverse \f{\mph[A]{h}^{-1}:\img{\mph[A]{h}}\rightarrow\mdl[A]{D}}. For each \f{\vv{v}\in\img{\mph[A]{h}}}, there exists a unique \f{a\in\mdl[a]{D}} such that:\begin{align*}
		\funeq{\mph[A]{h}}{a}{\vv{v}}.
	\end{align*}
	
	We define \mph[f]{h} as, for all \f{\vv{v\in\img{\mph[A]{h}}}}:\begin{align*}
		\funeq{\mph[f]{h}}{\vv{v}}{\fun{\mph[B]{h}}{\fun{f}{\fun{\mph[A]{h}^{-1}}{\vv{v}}}}}.
	\end{align*}
	
	\mph[f]{h} is well-defined on \img{\mph[A]{h}}, since the injectivity of \mph[A]{h} guarentees that \fun{\mph[A]{h}^{-1}}{\vv{v}} is uniquely defined for each \f{\vv{v}\in\img{\mph[A]{h}}}, and the function \mph[f]{h} assigns a unique value in \img{\mph[B]{h}} to each \f{\vv{v}\in\img{\mph[A]{h}}}.
	
	Let \f{a\in\mdl[A]{D}}. We compute \fun{\mph[f]{h}}{\fun{\mph[A]{h}}{a}}:\begin{align*}
		\fun{\mph[f]{h}}{\fun{\mph[A]{h}}{a}} & = \fun{\mph[B]{h}}{\fun{f}{\fun{\mph[A]{h}^{-1}}{\fun{\mph[A]{h}}{a}}}}\\
		& = \fun{\mph[B]{h}}{\fun{f}{a}},
	\end{align*}
	
	since \funeq{\mph[a]{h}^{-1}}{\fun{\mph[A]{h}}{a}}{a} by the injectivity of \mph[A]{h}. Commutativity of the diagram follows:
	
	\begin{itemize}
		\item path along \f{f} and \mph[B]{h}: \f{a \xrightarrow{f} \fun{f}{a} \xrightarrow{h_B} \fun{\mph[B]{h}}{\fun{f}{a}}}
		
		\item path along \mph[A]{h} and \mph[f]{h}: \f{a \xrightarrow{\mph[A]{h}} \fun{\mph[A]{h}}{a} \xrightarrow{\mph[f]{h}} \fun{\mph[f]{h}}{\fun{\mph[A]{h}}{a}}}
	\end{itemize}
	
	Since \f{\fun{\mph[f]{h}}{\fun{\mph[A]{h}}{a}} = \fun{\mph[B]{h}}{\fun{f}{a}}}, both paths lead to the same element in \img{\mph[B]{h}}.

\end{proof}

\begin{corollary}\label{cor:compo}
	For functions \f{f:\mdl[A]{D}\rightarrow\mdl[B]{D}} and \f{g:\mdl[B]{D}\rightarrow\mdl[C]{D}} in the extensional model \mdl[ext]{M} whose composition is given by \f{g \circ f: \mdl[A]{D} \rightarrow \mdl[C]{D}}, where \funeq{\bp{g \circ f}}{a}{\fun{g}{\fun{f}{a}}}, there exists a function \f{\mph[{g\circ f}]{h}:\img{\mph[A]{h}}\rightarrow\img{\mph[C]{h}}} in the vector space model \mdl[\mdl{S}]{M} such that, for all \f{a\in\mdl[A]{D}}:\begin{align*}
		\fun{\mph[{g\circ f}]{h}}{\fun{\mph[A]{h}}{a}} & = \fun{\mph[C]{h}}{\fun{\bp{g\circ f}}{a}}\\
		& = \fun{\mph[g]{h}}{\fun{\mph[f]{h}}{\fun{\mph[A]{h}}{a}}},
	\end{align*}

	where \f{\mph[A]{h}:\mdl[A]{D}\rightarrow\img{\mph[A]{h}}\subseteq\mdl[{\mdl[A]{D}}]{S}}, \f{\mph[B]{h}:\mdl[B]{D}\rightarrow\img{\mph[B]{h}}\subseteq\mdl[{\mdl[B]{D}}]{S}}, and \f{\mph[C]{h}:\mdl[C]{D}\rightarrow\img{\mph[C]{h}}\subseteq\mdl[{\mdl[C]{D}}]{S}}, are injective mappings from their respective domains to subsets of vector spaces in \mdl[\mdl{S}]{M}, and \f{\mph[f]{h}:\img{\mph[A]{h}}\rightarrow\img{\mph[B]{h}}} is defined by \funeq{\mph[f]{h}}{\fun{\mph[A]{h}}{a}}{\fun{\mph[B]{h}}{\fun{f}{a}}} for all \f{a\in\mdl[A]{D}}, \f{\mph[g]{h}:\img{\mph[B]{h}}\rightarrow\img{\mph[C]{h}}} is defined by \funeq{\mph[g]{h}}{\fun{\mph[B]{h}}{b}}{\fun{\mph[C]{h}}{\fun{g}{b}}} for all \f{b\in\mdl[B]{D}}. 

\end{corollary}

The diagram commutes:
\begin{figure}[H]
	\[\begin{tikzcd}[column sep=large, row sep=large] \mdl[A]{D} \arrow{r}{f} \arrow[swap]{d}{h_A} & \mdl[B]{D} \arrow{r}{g} \arrow{d}{h_B} & \mdl[C]{D} \arrow{d}{h_C} \\ 
	\img{\mph[A]{h}} \arrow[swap]{r}{h_f} & \img{\mph[B]{h}} \arrow[swap]{r}{h_g} & \img{\mph[C]{h}} \end{tikzcd}\]
	\end{figure}
	
	from which it follows that \f{\mph[{g\circ f}]{h} = \mph[g]{h}\circ \mph[f]{h}}.

\begin{proof}\label{pr:compo}
	The proof follows from direct computation, chasing the diagram in parts. First, for all \f{a\in\mdl[A]{D}}, we compute \fun{\mph[{g \circ f}]{h}}{\fun{\mph[A]{h}}{a}} by definition of \mph[{g\circ f}]{h} as:\begin{align*}
		\begin{aligned}
			\fun{\mph[{g \circ f}]{h}}{\fun{\mph[A]{h}}{a}} & = \fun{\mph[C]{h}}{\fun{\bp{g \circ f}}{a}} \\
			& = \fun{\mph[C]{h}}{\fun{g}{\fun{f}{a}}}.
		\end{aligned}
	\end{align*}
	
	Computing \fun{\mph[g]{h}}{\fun{\mph[f]{h}}{\fun{\mph[A]{h}}{a}}} proceeds in two parts. For \fun{\mph[f]{h}}{\fun{\mph[A]{h}}{a}}, we have by definition of \mph[f]{h}:\begin{align*}
		\funeq{\mph[f]{h}}{\fun{\mph[A]{h}}{a}}{\fun{\mph[B]{h}}{\fun{f}{a}}}.
	\end{align*} 
	
	Likewise, we compute \fun{\mph[g]{h}}{\fun{\mph[f]{h}}{\fun{\mph[A]{h}}{a}}}:\begin{align*}
		\begin{aligned}
			\fun{\mph[g]{h}}{\fun{\mph[f]{h}}{\fun{\mph[A]{h}}{a}}} & = \fun{\mph[g]{h}}{\fun{\mph[B]{h}}{\fun{f}{a}}} \\
			& =\fun{\mph[C]{h}}{\fun{g}{\fun{f}{a}}},
		\end{aligned}
	\end{align*}

	where the last equivalence follows by definition of \mph[g]{h}.
	
	From the above computations, we have:\begin{align*}
		\fun{\mph[{g \circ f}]{h}}{\fun{\mph[A]{h}}{a}} = \fun{\mph[C]{h}}{\fun{g}{\fun{f}{a}}} =  \fun{\mph[g]{h}}{\fun{\mph[f]{h}}{\fun{\mph[A]{h}}{a}}}.
	\end{align*}

	Therefore, for all \f{a\in\mdl[A]{D}},\begin{align*}
		\funeq{\mph[{g \circ f}]{h}}{\fun{\mph[A]{h}}{a}}{\fun{\mph[g]{h}}{\fun{\mph[f]{h}}{\fun{\mph[A]{h}}{a}}}},
	\end{align*}  and the diagram commutes.
\end{proof}

A generalization of this composition is given in Theorem~\ref{thm:compo2}. This demonstrates that the composition of functions in the extensional model corresponds to the composition of their mappings in the vector space model, and the diagram commutes for the entire composition.

\begin{theorem}\label{thm:compo2}
	
Let \f{n \geq 1} be an integer, and let \f{f_1, f_2, \dots, f_n} be a sequence of semantic functions in the extensional model \mdl[ext]{M}, where each \f{f_i} is a function \f{f_i: \mdl[{A_i}]{D} \rightarrow \mdl[{A_{i+1}}]{D}} for \f{i=1,2, \dots, n}. Then, for each \f{f_i}, there exists a function \f{\mph[f_i]{h}:\img{\mph[A_{i}]{h}}\rightarrow\img{\mph[A_{i+1}]{h}}} in the vector space model \mdl[{\mdl{S}}]{M} such that for all \f{a \in \mdl[{A_1}]{D}}:\begin{align*}
		\funeq{\bp{\mph[f_{n}]{h}\circ\mph[f_{n-1}]{h}\circ\dots\circ\mph[f_{1}]{h}}}{\fun{\mph[A_{1}]{h}}{a}}{\fun{\mph[A_{n+1}]{h}}{\fun{\bp{f_n \circ f_{n-1}\circ\dots\circ f_1}}{a}}}.
	\end{align*}

\end{theorem}

The diagram commutes:\begin{figure}[H]
	\[\begin{tikzcd}[column sep=large, row sep=large] D_{A_1} \arrow{r}{f_1} \arrow[swap]{d}{h_{A_1}} & D_{A_2} \arrow{r}{f_2} \arrow{d}{h_{A_2}} & \cdots \arrow{r}{f_n} & D_{A_{n+1}} \arrow{d}{h_{A_{n+1}}} \\ \img{\mph[A_{1}]{h}} \arrow[swap]{r}{h_{f_1}} & \img{\mph[A_{2}]{h}} \arrow[swap]{r}{h_{f_2}} & \cdots \arrow[swap]{r}{h_{f_n}} & \img{\mph[A_{n+1}]{h}} \end{tikzcd}\]
\end{figure}

\begin{proof}\label{pr:compo2}
	The proof follows by induction. Consider the base case, where for \f{n=1}, we recover what was shown in Lemma~\ref{lem:func}, in which we have a function \f{f_1:\mdl[A_1]{D}\rightarrow\mdl[A_2]{D}}. From Lemma~\ref{lem:func}, we know that there exists a function \f{\mph[f_1]{h}:\img{\mph[A_1]{h}}\rightarrow\img{\mph[A_2]{h}}} such that, for all \f{a\in\mdl[A_1]{D}}:\begin{align*}
		\funeq{\mph[f_{1}]{h}}{\fun{\mph[A_{1}]{h}}{a}}{\fun{\mph[A_{2}]{h}}{\fun{f_{1}}{a}}}.
	\end{align*}
	
	This establishes the base case.
	
	Assume now that Theorem~\ref{thm:compo2} holds for some \f{n = k\geq 1}; that is, for functions \f{f_1,f_2,\dots,f_k}, we have, for all \f{a\in\mdl[A_1]{D}}:\begin{align*}
		\funeq{\bp{\mph[f_{k}]{h}\circ\mph[f_{k-1}]{h}\circ\dots\circ\mph[f_{1}]{h}}}{\fun{\mph[A_{1}]{h}}{a}}{\fun{\mph[A_{k+1}]{h}}{\fun{\bp{f_k \circ f_{k-1}\circ\dots\circ f_1}}{a}}}.
	\end{align*}
	
	Now consider an additional function \f{f_{k+1}:\mdl[A_{k+1}]{D}\rightarrow\mdl[A_{k+2}]{D}}. By Lemma~\ref{lem:func}, there exists a function \f{\mph[f_{k+1}]{h}:\img{\mph[A_{k+1}]{h}}\rightarrow\img{\mph[A_{k+2}]{h}}} such that, for all \f{b\in\mdl[A_{k+1}]{D}}:\begin{align*}
		\funeq{\mph[f_{k+1}]{h}}{\fun{\mph[A_{k+1}]{h}}{b}}{\fun{\mph[A_{k+2}]{h}}{\fun{f_{k+1}}{b}}}.
	\end{align*}
	
	Now, consider the composition up to \f{n=k+1}:\begin{align*}
		\bp{\mph[f_{k+1}]{h}\circ\mph[f_{k}]{h}\circ\mph[f_{k-1}]{h}\circ\dots\circ\mph[f_{1}]{h}}:\img{\mph[A_1]{h}\rightarrow\img{\mph[A_{k+2}]{h}}}.
	\end{align*}
	
	We apply this composition to \fun{\mph[A_1]{h}}{a}:\begin{align*}
		\begin{aligned}
			\fun{\bp{\mph[f_{k+1}]{h}\circ\mph[f_{k}]{h}\circ\mph[f_{k-1}]{h}\circ\dots\mph[f_{1}]{h}}}{\fun{\mph[A_{1}]{h}}{a}} & =\fun{\mph[f_{k+1}]{h}}{\fun{\bp{\mph[f_k]{h}\circ\mph[f_{k-1}]{h}\circ\dots\circ\mph[f_1]{h}}}{\fun{\mph[A_{1}]{h}}{a}}} & \\
			& =\fun{\mph[f_{k+1}]{h}}{\fun{\mph[A_{k+1}]{h}}{\fun{\bp{f_k\circ f_{k-1}\circ\dots\circ f_1}}{a}}}  & \qquad\esc{(by induction hypothesis)}\\
			& =\fun{\mph[A_{k+2}]{h}}{\fun{f_{k+1}}{\fun{\bp{f_k \circ f_{k-1}\circ\dots\circ f_1}}{a}}} & \qquad\esc{(by definition of \f{\mph[f_{k+1}]{h}})}\\
			& =\fun{\mph[A_{k+2}]{h}}{\fun{\bp{f_{k+1}\circ f_k \circ f_{k-1}\circ\dots\circ f_1}}{a}}, &
		\end{aligned}
	\end{align*}
	
	Where the first equality is a direct application of the composition to \fun{\mph[A_1]{h}}{a}; the second equality follows from the inductive hypothesis; the third equality follows from substituting \f{b = \fun{\bp{f_k\circ f_{k-1}\circ\dots\circ f_1}}{a}} for the definition of \mph[f_{k+1}]{h}, where, for all \f{b\in\mdl[A_{k+1}]{D}}, \funeq{\mph[f_{k+1}]{h}}{\fun{\mph[A_{k+1}]{h}}{b}}{\fun{\mph[A_{k+2}]{h}}{\fun{f_{k+1}}{b}}}; the fourth equality simplifies to the composition of functions up to \f{f_{k+1}}.
	
	Thus, we have shown that:\begin{align*}
		\fun{\bp{\mph[f_{k+1}]{h}\circ\mph[f_{k}]{h}\circ\mph[f_{k-1}]{h}\circ\dots\mph[f_{1}]{h}}}{\fun{\mph[A_{1}]{h}}{a}} = \fun{\mph[A_{k+2}]{h}}{\fun{\bp{f_{k+1}\circ f_k \circ f_{k-1}\circ\dots\circ f_1}}{a}}.
	\end{align*}
	
	Therefore, Theorem~\ref{thm:compo2} holds for \f{n=k+1}, and by induction, for all integers \f{n\geq 1}, the diagram commutes. At each function \f{f_i} for all \f{a_i\in\mdl[A_i]{D}}, we have that:\begin{align*}
		\funeq{\mph[f_{i}]{h}}{\fun{\mph[A_{i}]{h}}{a_i}}{\fun{\mph[A_{i+1}]{h}}{\fun{f_{i}}{a_i}}},
	\end{align*}

	and the composition \f{\mph[f_n]{h}\circ\mph[f_{n-1}]{h}\circ\dots\circ\mph[f_1]{h}} corresponds to the composition \f{f_n\circ f_{n-1}\circ\dots\circ f_1}. For all \f{a\in\mdl[A_1]{D}}:\begin{align*}
		\funeq{\bp{\mph[f_{n}]{h}\circ\mph[f_{n-1}]{h}\circ\dots\circ\mph[f_{1}]{h}}}{\fun{\mph[A_{1}]{h}}{a}}{\fun{\mph[A_{n+1}]{h}}{\fun{\bp{f_n \circ f_{n-1}\circ\dots\circ f_1}}{a}}},
	\end{align*}
	
	in which the path from \mdl[A_1]{D} to \mdl[A_{n+1}]{D} via \f{f_n \circ f_{n-1}\circ \dots \circ f_1} and mapped by \mph[A_{n+1}]{h} is equivalent to the path from \mdl[{A_1}]{D} mapped by \mph[A_1]{h} and then through the sequence of mappings \f{\mph[f_1]{h},\mph[f_2]{h},\dots,\mph[f_n]{h}}.

\end{proof}

\section{Principle of compositionality}\label{sec:comp}
Compositionality is the principle that an expression's meaning derives from its parts and their combination; it addresses how language generates infinite structures from finite resources. Though often attributed to Frege\footnote{\cite{baronietal2014frege} mention that Frege never actually explicitly stated this principle, which was likely assumed in Boole's earlier work.}, the concept's roots trace to Plato, who argued that sentences have structure, parts with different functions, and meaning determined by these parts. The principle gained mathematical formalization in discrete logic \cite{Boole1854,Frege1884}, developing in parallel to work on continuous vector spaces by \cite{Hamilton1847,Grassmann1862}.

\cite{Montague1970} later formalized compositionality as a homomorphism between linguistic expressions and meanings, modeling both syntax and semantics as multi-sorted algebras. His \qte{Universal Grammar} posited no fundamental difference between natural and logical languages, introducing systematic natural language processing through what became known as Montague grammar. This approach requires two elements: an explicit lexicon providing logical forms for words, and explicit rules determining combination order for valid semantic representation. The resulting syntax-driven semantic analysis assumes a mapping from syntactic to semantic types, where syntactic composition implies semantic composition (the rule-to-rule hypothesis).

A \term{multi-sorted algebra} is an algebraic structure with multiple sorts (types), each with its own domain of elements and operations. First, the source algebra represents syntactic structures. Let \f{S} be a set of syntactic sorts (\eg, Noun Phrase (NP), Verb Phrase (VP), Sentence (S)). For each sort \f{s\in S}, \f{A_s} is the set of syntactic elements (expressions) of sort \f{s}. Let \f{\Gamma} be a set of syntactic operation symbols, where, for each \f{\gamma \in \Gamma}, there is an operation:\begin{align*}
	F_\gamma: A_{{\fun{s_1}{\gamma}}}\times A_{{\fun{s_2}{\gamma}}}\times\dots\times A_{{\fun{s_n}{\gamma}}} \rightarrow A_{{\fun{s}{\gamma}}},
\end{align*}

where \f{\fun{s_1}{\gamma},\fun{s_2}{\gamma},\dots,\fun{s_n}{\gamma},\fun{s}{\gamma}\in S}.

\begin{definition}\label{def:alg}
	The \term{source algebra} is defined as:\begin{align*}
		\mdl{A} = \bt{\bp[{s\in S}]{A_s},\bp[{\gamma\in \Gamma}]{F_\gamma}},
	\end{align*}
\end{definition}

which, in prose-style:\begin{align*}
	\mdl{A} = \bt{\esc{`}\text{\makecell{collection of sets of syntactic elements \\ of specific syntactic categories}}\esc{'},\esc{`}\text{set of syntactic operations}\esc{'}}.
\end{align*}

Second, the target algebra represents semantic interpretations. Let \f{T} be a set of semantic sorts (which we saw as types before). For each sort \f{t\in T}, \f{B_t} is the set of semantic elements of sort \f{t}. Let \f{\Delta} be a set of semantic operations symbols such that, for each \f{\delta\in\Delta}, there is an operation:\begin{align*}
	G_\delta: B_{{\fun{t_1}{\delta}}}\times B_{{\fun{t_2}{\delta}}}\times\dots\times B_{{\fun{t_n}{\delta}}} \rightarrow B_{{\fun{ts}{\delta}}},
\end{align*}

where \f{\fun{t_1}{\delta},\fun{t_2}{\delta},\dots,\fun{t_n}{\delta},\fun{t}{\delta}\in T}.

\begin{definition}\label{def:sem}
	The \term{target algebra} is defined as:\begin{align*}
		\mdl{B} = \bt{\bp[{t\in T}]{B_t},\bp[{\delta\in \Delta}]{G_\delta}}.
	\end{align*}
\end{definition}

In prose-style:\begin{align*}
	\mdl{B} = \bt{\esc{`}\text{\makecell{collection of sets of semantic elements \\ of specific semantic categories}}\esc{'},\esc{`}\text{set of semantic operations}\esc{'}}.
\end{align*}

An interpretation, then, is a homomorphism \mph{h} from the source algebra \mdl{A} to the target algebra \mph{B}. This homomorphism preserves the algebraic structure such that syntactic combinations correspond to semantic combinations. That is, there is a mapping \f{\sigma: S\rightarrow T} from syntactic sorts to semantic sorts such that, for all \f{a\in A_s}:\begin{align*}
	\fun{h}{a}\in B_{\fun{\sigma}{s}},
\end{align*}

such that elements of a syntactic category are mapped to elements of the corresponding semantic type. There is a mapping \f{\rho:\Gamma\rightarrow\Delta} from syntactic operation symbols to semantic operation symbols such that, for all \f{\gamma \in \Gamma} and \f{a_i \in A_{\fun{s_i}{\gamma}}}:\begin{align*}
	\funeq{h}{\fun{F_\gamma}{a_1,a_2,\dots,a_n}}{\fun{G_{\fun{\rho}{\gamma}}}{\fun{h}{a_1},\fun{h}{a_2},\dots,\fun{h}{a_n}}}.
\end{align*}

This equation \emph{is} compositionality: the meaning of a complex expression is determined by the meanings of its parts and the syntactic operation used to combine them. This compositionality is driven exactly by explicit syntax and explicit semantics. How, then, is a distributional treatment accounted for? Additionally, the productivity argument supports compositionality's validity: humans understand novel sentences using only word meanings and combination rules. However, while natural even to young children, it is the \emph{computational implementation} of composition that has proven non-trivial. By the twenty-first century, the field had developed both distributional vector-based models (capturing lexical semantics through usage patterns) and symbolic compositional models; their integration, if it exists, remained an open question.

This paper is \emph{not} a paper about unification; rather, we show a \emph{compatibility} between the two semantics. This is an important distinction. Humans have available to them logical reasoning, but also relational reasoning. It stands to reason that multiple models of language semantics should account for these, and it is natural to inquire how these models may be accessible one to the other. What we have shown is, assuming a vector logic, that language processing is well-founded for this multimodal semantics; in this way, compositionality is shown in a distributional semantics, influenced by composition in a formal semantics. Here, we have supplied proofs for this compatibility, which have been absent elsewhere.

We have shown the compatibility exists; the \emph{character} of the composition in distributional semantics, however, is variable, and much work has been done here. Phrase vectors are built over word vectors; algebraic models of composition follow from a number of possible approaches. Vectors by themselves do not necessarily respect compositionality; the critical failure here is a consequence of the definition of a vector space. While closure under vector addition permits the (linear) combination of linguistic constituents, language is not necessarily linear; commutativity and associativity of vector addition are responsible for possible interpretations:\begin{align*}\label{eqn:bad}
	\vv{the\, student\, read\, the\, book} & =\vv{the}+\vv{student}+\vv{read}+\vv{the}+\vv{book} \\
	& =\vv{the}+\vv{student}+\vv{read}+\vv{book}+\vv{the} \\
	& = \dots \\
	& =\vv{read}+\vv{student}+\vv{the}+\vv{the}+\vv{book} \\
	& = \dots \\
	& = \vv{book}+\vv{read}+\vv{student}+\vv{the}+\vv{the} \\
	& = \dots
\end{align*}

Extending distributional models to respect language at the phrase-level has been a hard-fought conflict, accounted in \cite{shr2010,baroni2010nouns,mitchell2010composition,socher2012semantic,baronietal2014frege}, building on earlier work from cognitive science \cite{foltz1996latent,kintsch2001predication}. Compositional Distributional Semantics, especially equipped with category theory, has shown a successful integration of the Lambek calculus-based Pregroup grammar with vector spaces \cite{shr2010}. Attempts at composition by addition \cite{seonwooetal2019additive,naito2022revisitingadditivecompositionalityand,tian2017mechanism}, multiplication \cite{kiros2014multiplicative,luan2016multiplicative,mitchelllapata2010composition}, multiplication with weighted sums \cite{mitchelllapata2008vector,mitchelllapata2010composition}, tensor products \cite{kartsaklis2013prior,milajevsetal2014evaluating,smolensky1990tensor}, and otherwise mixture-based models \cite{baronietal2014frege} have been explored. Furthermore, representing function words has been a challenge, despite efforts such as those in \cite{herbelotvecchi2015building,grefenstette2013towards,hermann2013not}. Distributional semantics is good at modeling semantics intuition, capturing conceptual or generic knowledge, but not necessarily referential (extra- or intra-linguistic) information. Applying distributional models to entailment tasks can certainly increase coverage \cite{westera2019dont}, but at the cost of precision \cite{beltagy2013montague}; indeed, entailment decisions require predicting whether an expression applies to the same referent or not, and these systems make trivial mistakes.

Instead, what we have shown here by Theorem~\ref{thm:homo} is that compositionality in formal semantics and compositionality in distributional semantics are \emph{compatible} one with the other, \emph{not} unified as has been attempted elsewhere; that is, there is \emph{not} an isomorphism (in fact, if we force an isomorphism, then we run into issues about coverage and infinite sets), but a homomorphism. This compatibility, derived from the recursion shown in Theorem~\ref{thm:recur}, for the well-foundedness in the homomorphisms of \f{H = \bc[\tau\in T]{\mph[\tau]{h}}}, and the \f{n}-ary composition in Theorem~\ref{thm:compo2}, shows that composition may be well-defined in a distributional model. The characterization of compatibility here supports hybrid cognitive models that combine symbolic and sub-symbolic processing, and so suggests that we have access to mental representations both compositional and distributed, not mutually exclusive. 

Formal semantics provides models for truth-conditional meaning; vector space semantics provide models that capture contextual and relational meaning. The former aligns with Fregean and Tarskian \qte{meaning as reference}; the meaning of an expression is its correspondence with entities, relations, or truth values in a model. The latter follows Wittgenstein's \qte{meaning as use} perspective; meaning, then, is inferred from the distributional properties of words, and so is sensitive to usage patterns and conceptual similarity rather than specific referential grounding. Humans have access to both, and it is their compatibility that suggests a multimodal processing.

\section{A note on intensionality}\label{sec:intens}
Where the extensional interpretation of a function is concerned only with the input-output, the \term{intensional}\footnote{Note that it is intentional that this word is spelled \orth{intensional}.} interpretation of a function is concerned with \emph{how} those functions are computed; two functions \f{f:X\rightarrow Y} and \f{g:X\rightarrow Y} with identical input-output mappings may differ intensionally, \ie, they have different algorithms or formulae\footnote{The ur-example is from computer science, in which functions are formulae: some sorting algorithm \f{A} and another \f{B} can have \emph{very} different algorithm designs with \emph{very} different complexities, regardless if they have the same sorted list output when give the same input.}. Functions \f{f:X\rightarrow Y} and \f{g:X\rightarrow Y} are considered identical only if they are defined by essentially the same formula \cite{Selinger2013}. Intensional semantics is concerned with parametrizations pointing to elements of a set (that is, there is an entire additional functional step before accessing some element of a set); the intension of an expression is a function that takes a parameter as input and returns the extension of the expression \cite{quigley2024categoricalframeworktypedextensional}. 

Intensional semantics has the three primitive types: \bc{e,t,s}. Here: \f{e} is the type of \term{entities}, all entities have type \f{e}; \f{t} is the type of \term{truth values}, the truth values \bc{1,0} have type \f{t}; \f{s} is a type of \term{index}; indices may be any of intensional domains (worlds, times, world-time pairs, \etc, which may also be specified for types themselves).\begin{align*}
	e &: \text{entities in the world} \\
	t &: \text{truth or false} \in \bc{1,0} \\
	s &: \text{circumstantial indices}
\end{align*}

\begin{definition}[Intensional Model]\label{def:intenmod}
	Let \f{\mathcal{T}} be a set of types, and \f{\mathcal{K} = \mathcal{T}\setminus\bc{e,t}} be a set of types \sans entities and truth values. An \term{intensional model} \mdl[int]{M} is a triple of the form:\begin{align*}
		\mdl[int]{M} = \bt{\bp[\tau\in\mathcal{T}\setminus\mathcal{K}]{\mdl[\tau]{D}},\bp[\tau\in\mathcal{K}]{\mdl[\tau]{F}},\mdl{I}},
	\end{align*}
	where:
	\begin{itemize}
		\item for each type \f{\tau\in\mathcal{T}\setminus\mathcal{K}}, the model determines a corresponding domain \mdl[\tau]{D}. The \term{standard domain} is an indexed family of sets \bp[\tau\in\mathcal{T}\setminus\mathcal{K}]{\mdl[\tau]{D}}.
		\begin{itemize}
			\item \f{\mdl[e]{D}=\mdl{D}} is the set of entities.
			\item \f{\mdl[t]{D}=\bc{0,1}} is the set of truth-values.
		\end{itemize}
		\item For each type \f{\tau\in\mathcal{K}}, the model determines a corresponding domain frame \mdl[\tau]{D}. The \term{Kripke frame} is an indexed family of tuples whose elements are a set and a relation on that set \bp[\tau\in\mathcal{K}]{\mdl[\tau]{F}}.
		\begin{itemize}
			\item \bp[\tau\in\mathcal{K}]{\mdl[\tau]{D}} is an indexed family of non-empty (inhabited) sets of elements
			\item \bp[\tau\in\mathcal{K}]{\mdl[\tau]{R}} is an indexed family of binary accessibility or transition relations \f{\mdl{R}:\bp[\tau\in\mathcal{K}]{\mdl[\tau]{D}}\times\bp[\tau\in\mathcal{K}]{\mdl[\tau]{D}}}
		\end{itemize}
		
		\item The set of indices \f{S} is the cartesian product of the domains for intensional tyes: \f{S = \prod_{\tau\in\mdl{K}} \mdl[\tau]{D}}. Each element \f{s\in S} is an index, representing a combination of values from the intensional domains.

		For each type \f{\tau\in\mathcal{T}} and each non-logical constant \f{c_\tau} of type \f{\tau} in the formal language \fl, the intensional \term{interpretation function} \mdl{I} assigns an intension, which is a function from indices \f{s} to elements of the corresponding domain.
		
		Let \f{V} be the set of individual variables, and \mdl[e]{D} be the domain of entities. The \term{assignment function} \f{\gsn:V\rightarrow\mdl[e]{D}} maps each variable \f{x\in V} to an entity \f{\fun{\gsn}{x}\in\mdl[e]{D}}. For any variable \f{x\in V} and individual \f{k\in \mdl[e]{D}}, the \term{variant assignment function} \gsnt{x\mapsto k} is defined by:\begin{align*}
			\fun{\gsnt{x\mapsto k}}{y}= \begin{cases}k & \text { if } y=x, \\ \fun{\gsn}{y} & \text { if } y \neq x .\end{cases}
		\end{align*}

		\begin{enumerate}
			\item	If \sph{} is a constant \f{\sph \in \fl}, \fun{\mdl{I}}{\sph} is a function from indices \f{S} to elements of the domain \mdl[\tau]{D} by \f{\fun{\mdl{I}}{\sph}:S\rightarrow \mdl[\tau]{D}}. Thus, at each index \f{s\in S}, \f{\funt{\mdl{I}}{\sph}{s}\in\mdl[\tau]{D}}.

			\item If \sph{} is a variable \f{\sph \in V\com \funeq{\gsn}{\sph}{k} \in \mdl[e]{D}}.

			\item For any \f{n}-ary predicate \f{P \in\fl}, the interpretation \fun{\mdl{I}}{P} is a function \f{\fun{\mdl{I}}{P}:S\rightarrow \fun{\mdl{P}}{\mdl[e]{D}^{n}}}, where \fun{\mdl{P}}{\mdl[e]{D}^{n}} is the power set of \f{\mdl[e]{D}^{n}}. At each index \f{s\in S}, \f{\funt{\mdl{I}}{P}{s}\subseteq\mdl[e]{D}^n} represents the set of \f{n}-tuples satisfying \f{P} at \f{s}.

			\item For any \f{n}-ary function \f{f\in \fl}, the interpretation \fun{\mdl{I}}{f} is a function: \f{\fun{\mdl{I}}{f}:S\rightarrow\bp{\mdl[e]{D}^n \rightarrow \mdl[e]{D}}}. At each index \f{s\in S}, \f{\funt{\mdl{I}}{f}{s}:\mdl[e]{D}^n \rightarrow \mdl[e]{D}} is a function from \f{n}-tuples of entities to entities.

		\end{enumerate}
		
		An intensional \term{denotation function} \den[\mdl{M},\gsn,s]{.} assigns to every expression \sph{} of the language \fl{} a semantic value \den[\mdl{M},\gsn,s]{\sph} by recursively building up basic interpretation functions \funt{\mdl{I}}{.}{s} as \funt{\mdl{I}}{\sph}{s}. For all indices \f{s\in \prod_{\tau\in\mathcal{K}} \mdl[\tau]{D}}:
		
		\begin{enumerate}
			\item if \sph{} is a constant \f{\sph \in \fl\com \den[\mdl{M},\gsn,s]{\sph} = \funt{\mdl{I}}{\sph}{s}}.
			
			\item if \sph{} is a variable in a set of variables \f{\sph \in V\com \den[\mdl{M},\gsn,s]{\sph} = \fun{\gsn}{\sph}}.
			
			\item for any \f{n}-ary predicate \f{P \in\fl} and sequence of terms \f{\sph_1,\sph_2,\dots,\sph_n}, we have that \f{\den[\mdl{M},\gsn,s]{\fun{P}{\sph_1,\sph_2,\dots,\sph_n}}=1\wenn}\begin{align*}
				\bp{\den[\mdl{M},\gsn,s]{\f{\sph_1}},\dots,\den[\mdl{M},\gsn,s]{\f{\sph_n}}}\in\den[\mdl{M},\gsn,s]{\f{P}}=\funt{\mdl{I}}{P}{s}.
			\end{align*}						
			
			\item for any \f{n}-ary function \f{f \in\fl} and sequence of terms \f{\sph_1,\sph_2,\dots,\sph_n}, we have that\begin{align*}
				\den[\mdl{M},\gsn,s]{\fun{f}{\sph_1,\dots,\sph_n}} & = \fun{\den[\mdl{M},\gsn,s]{\f{f}}}{\den[\mdl{M},\gsn,s]{\f{\sph_1}},\dots,\den[\mdl{M},\gsn,s]{\f{\sph_n}}} \\
				& =\fun{\funt{\mdl{I}}{f}{s}}{\den[\mdl{M},\gsn,s]{\f{\sph_1}},\dots,\den[\mdl{M},\gsn,s]{\f{\sph_n}}}.
			\end{align*} 
			
		\end{enumerate}
		
	\end{itemize}

\end{definition}

The technicalities of Definition~\ref{def:intenmod} may be parsed and characterized in prose-terms as:\begin{align*}
	\mdl[int]{M} = \bt{\esc{`}\text{\makecell{collection of type-specific \\ sets of allowed values}}\esc{'}, \esc{`}\text{\makecell{binary relations on collection of \\  type-specific sets of allowed values}}\esc{'},\esc{`}\text{\makecell{function that gives \\ meaning to symbols}}\esc{'}}.
\end{align*}

It remains, then, to extend Theorem~\ref{thm:homo} to now accommodate the relationship between frame structures in the formal model with their representation in the vector space model. The intensional model of Definition~\ref{def:intenmod} has this additional structure of the frame relative to the extensional model of Definition~\ref{def:extmod}; how we represent this additional structure in vector logic remains. 

Some work in this area follows. Kripke frames find representation in vector spaces through $K$-algebras\footnote{Some examples: quaternions \bp{H,\star_H}; octonions \bp{O,\star_O}; matrix algebras \bp{M_n,\star} and \bp{M_n,\circ_J}.} \cite{greco2021vectorspaceskripkeframes} \bp{V,\star_V}, where the vector space's subspace structure \fun{S}{V} forms a complete sub-$\bigcap$-semilattice of \fun{S}{V}. A closure operator \bq{-} acts as a nucleus on the power set, while bilinear operations encode ternary relations. This structure generates a complete residuated lattice \f{V^+}, with modal extensions achieved through relations \f{\mdl{R}\subseteq V\times V} compatible with scalar multiplication. The framework extends to quantum mechanical formalisms via isomorphisms between quantum Kripke frames and irreducible Hilbertian geometries \cite{Zhong2018}, and generalizes further through embedding into pointed stably supported quantales \cite{MARCELINORESENDE2008}, where modal operators are interpreted through the action of accessibility elements on valuations. 

While our work employs the lambda calculus as its foundational framework, \cite{palm2024} demonstrates how vector spaces can function as Kripke-style models where the subspace lattice relates to the Boolean algebra component of complex algebras in traditional frames. In their approach, bilinear operations serve as ternary relations supporting multiplicative Lambek calculus semantics, exactly contrasting with our lambda calculus-based framework. Their binary operations on subspace lattices emerge analogously to Routley-Meyer semantics connectives.

The engine, therefore, that makes this representation work throughout the above is driven by a Lambek calculus. This is an important detail. The Lambek calculus representation naturally aligns with vector semantics through its categorical and resource-sensitive nature, where the tensor product \f{\otimes} directly models linear composition and the residuated lattice structure captures resource consumption. If we are to adapt Theorem~\ref{thm:homo} as we have worked out here to Kripke frame structures, we must explicitly handle variable binding and substitution, which are not necessarily native to the vector space Kripke frames as presented. A key component of our proofs throughout has been the type-drivenness of the models; function types beyond the multiplicative structure of Lambek calculus requires, therefore, an extension of the $K$-algebra structure to model functional application and abstraction. Finally, management of non-linear resources is evergreen in the Heim and Kratzer style semantics; lambda calculus permits unrestricted variable use unlike Lambek calculus's resource sensitivity, which we need to adapt appropriately for any meaningful semantics to result.

\section{Conclusion and perspectives}\label{sec:disc}
This work demonstrates a compatibility between formal and distributional semantics through vector logic, in which we establish well-defined correspondences between logical and matrix operations. In this way, we explicitly reject a unification of the two semantics in favor of brokering compatibility, in no small part informed by critiques of over-mathematization otherwise in mathematical sciences \cite{GYLLINGBERG2023109033}. This is a work of \emph{mathematical linguistics}, and seeks to explain how information is represented in natural languages via structures\footnote{No real citation with this description, but this characterization of mathematical linguistics follows from a conversation with Artem Burnistov, 2023.}; \ie, we study the linguistic structures (of semantics), how they relate to each other, and the extent to which they may be accessible one to the other. Contrast this with, say, \emph{computational linguistics}, which seeks to study a particular \emph{process} and solve a problem therein. Studying language structures through the lens of mathematical linguistics, we may get the process component from a computational perspective for free (to a certain extent), in which we provide a well-defined foundation upon with to elaborate and develop the compatibility between the two semantics. The connective tissue, the compatibility, follows from a vector logic. In this way, we see the same sky, just through a different lens. 

We make no hard claims like Frege (how natural language relates to formal language), and we do not take such an extreme position as Chomsky (that natural language and formal language are entirely inequivocal); rather, we work in the assumption that mathematics \emph{at least} approximates processing and is useful (enough) for understanding and studying language processes, and \emph{at most} indicative of general mental capability. Science, after all, is the art of approximation; how acceptable that approximation may be differs by domain. Isaac Newton never claimed that the universe is calculus; the universe is not calculus, though calculus may model it. The magician should exercise caution that he is not deceived by his own illusions.

We stress caution: the pursuit of unification often leads to:

\begin{enumerate}
	\item an over-emphasis on mathematical elegance at the expense of empirical adequacy;\footnote{As Krugman notes of economics: \qte{economists mistook beauty, clad in impressive-looking mathematics, for truth} \cite{krugman2009how}.}
	\item focus on simplified models that fail to capture system-specific complexity;
	\item neglect of actual phenomena in favor of theoretical generality.
\end{enumerate}

Linguistics is a science; it is based on the interpretation of experiments. Mathematics serves a dual role here: (1) it is at least a tool for presenting, analyzing, and forecasting language data (an apparatus for the art of approximation); (2) at its most ambitious, it may indicate general mental processes. However which way we tend towards, we must exercise caution: language theories \emph{must} be anchored in empirical evidence and require experimental validation to be credible; linguistics is not esoteric diagramming. 

Instead, we advocate for a compatibility, a perspective of a kind of pluralism that offers:

\begin{enumerate}
	\item recognition of distinct but interrelated semantic frameworks that are accessible one to the other;
	\item mathematical rigor grounded in observed language processing (here, that is reasoning based on logical processes and reasoning based on distributional properties);
	\item preservation of domain-specific insights while enabling formal correspondence.
\end{enumerate}

Semantic pluralism acknowledges that diverse frameworks serve distinct purposes while maintaining principled connections between them, reflecting our cognitive ability to switch between different modes of reasoning and semantic processing. Rather than seeking a unified theoretical framework, which would oversimplify language's inherent complexity, we embrace an open-ended approach that continuously explores new perspectives, providing mathematical foundations for understanding compositional distributional semantics without constraining it to a finite set of formalisms. To this end, we have provided here a well-defined foundation of relating models for what might have otherwise been taken for granted in studying a compositional distributional semantics.

\newpage
\bibliography{references}

\end{document}